\pdfoutput=1
\documentclass[nonblindrev]{informs3_hide}

\usepackage[colorlinks=true,citecolor=blue]{hyperref}
\usepackage{paralist}
\usepackage{cleveref}
\usepackage{threeparttable} % Ensure this package is included
\usepackage{array} % For new column types
\usepackage{natbib}
 \bibpunct[, ]{(}{)}{,}{a}{}{,}%
 \def\bibfont{\small}%
\usepackage[x11names]{xcolor}
\usepackage{algorithm}
\usepackage{algpseudocode}
\usepackage{setspace}
\usepackage{bm}
\usepackage{tabu}
\usepackage{thmtools}
\usepackage{thm-restate}
\usepackage{array}
\usepackage{multirow}
\usepackage{comment}
\usepackage[toc,page]{appendix}
\usepackage{appendix}
\usepackage{setspace}
\usepackage{amsmath}
\usepackage{dsfont}
\usepackage{tabularx}
\usepackage{hyperref}
\usepackage{url}
\usepackage{bbm}
\usepackage{booktabs}
\usepackage{enumitem}
\usepackage{float}
\usepackage{indentfirst}
\usepackage{adjustbox}
\usepackage{booktabs}
\usepackage{longtable}
\usepackage{array}
\TheoremsNumberedThrough     % Preferred (Theorem 1,
\ECRepeatTheorems
\EquationsNumberedThrough    % Default: (1), (2), ...

\newcommand{\xhdr}[1]{\vspace{1mm} \noindent{\bf #1}}

\newcommand{\mcsdf}{\texttt{MC-SF}}
\newcommand{\mcsdff}{\texttt{MC-SF$_2$}}
\newcommand{\sortF}{\texttt{Sorted-F}}
\newcommand{\sortLP}{\texttt{Sorted-LP}}
\newcommand{\opt}{\texttt{Optimal}}
\newcommand{\LPSwap}{\texttt{LP-Swap}}

\begin{document}

\ARTICLEAUTHORS{
\AUTHOR{Meixuan Wang}\AFF{Department of Computer Science and Technology, Tsinghua University,
\EMAIL{wangmx22@mails.tsinghua.edu.cn}}
\AUTHOR{Yinyu Ye}
\AFF{Department of Management Science and Engineering, Stanford University \&  Department of Industrial Engineering and Decision Analytics, HKUST, \EMAIL{yyye@stanford.edu}}
\AUTHOR{Zijie Zhou}\AFF{Department of Industrial Engineering and Decision Analytics, HKUST,
\EMAIL{jerryzhou@ust.hk}}
}

\RUNAUTHOR{}
\RUNTITLE{}
\TITLE{LLM Serving Optimization with Variable Prefill and Decode Lengths}
\ABSTRACT{

\noindent\textbf{Problem definition:}
We study offline scheduling for large language model (LLM) serving under a fixed KV-cache memory budget, where requests have heterogeneous prompt (prefill) and response (decode) lengths. Prompt tokens determine initial KV-cache usage, and each generated token further increases memory consumption. Given a backlog of n requests arriving together, the scheduler forms mixed prefill and decode batches over time to minimize total end-to-end latency, equivalently, the sum of completion times.

\noindent\textbf{Methodology/results:}
We show that heterogeneous prompt lengths make the scheduling problem fundamentally harder: the problem is NP-hard, and standard policies such as first-come-first-served, shortest-output-first, and total-size-based prioritization can have unbounded approximation ratios. We propose \sortF, which repeatedly forms feasible batches using an F-metric that balances batch size against downstream decode cost. 
% We prove that Sorted-F achieves a constant-factor approximation guarantee in the offline/backlogged model. We also develop practical implementations, including an exact dynamic programming for small instances and scalable local search and greedy heuristics for larger instances. 
We prove a factor-$40$ approximation guarantee for exact $F$-based batch selection in the unit-time decoding model with known output lengths. The guarantee applies to every finite instance and also holds when batch selection is restricted by a static peak-memory constraint. We develop an exact dynamic program for this static subproblem and scalable local-search and greedy heuristics for larger instances.
Experiments on public workloads that combine short conversations and long-document summarization show that F-metric-based scheduling consistently reduces average latency relative to standard baselines and remains close to the LP relaxation lower bound for tractable instances.

\noindent\textbf{Managerial implications:}
LLM serving systems often experience bursty demand, and peak periods can create backlogs that must be prioritized under tight KV-cache constraints. Our results show that common rules, such as greedily packing the GPU or prioritizing short responses, can perform poorly when prompt lengths vary widely, as in retrieval-augmented generation, document summarization, and multi-turn dialogue. The proposed F-metric provides a simple batching criterion that balances memory utilization against future decode congestion, while the algorithmic variants offer runtime-quality trade-offs for production scheduler design and capacity planning.

% \noindent We study offline scheduling for large language model (LLM) serving under a fixed KV-cache memory budget, where requests have heterogeneous prompt (prefill) and response (decode) lengths. Prompt tokens determine initial KV usage, and each generated token increases memory by one unit. Given a backlog of n requests arriving together, we schedule mixed prefill and decode batches to minimize total end-to-end latency. We show that heterogeneity in prompt lengths makes the problem computationally intractable and that widely used heuristics such as first-come-first-served and shortest-first can be arbitrarily suboptimal. We propose \sortF, which repeatedly forms feasible batches using a new selection metric that balances batch size against downstream decode cost, and prove it achieves a constant-factor guarantee on total latency. We further develop practical variants---an exact solver for small instances and fast heuristics for larger ones---and evaluate them on a public workload spanning short conversations and long-document summarization, where they consistently reduce average latency relative to standard baselines. Our results highlight that during bursty peak-hour backlogs, greedy GPU packing or short-request prioritization can perform poorly when prompt lengths vary widely, and provide a principled, tunable framework for designing production batch schedulers and planning capacity in memory-constrained LLM serving systems.
}
\maketitle

\section{Introduction} \label{sec:intro}

Modern large-scale language models \citep{brown2020language,openai2023gpt} have revolutionized artificial intelligence by demonstrating unprecedented capabilities in natural language generation across diverse linguistic domains and situational contexts. These sophisticated neural networks, trained on extensive corpora of textual data, now serve as foundational components for numerous real-world applications. Their deployment spans conversational agents \citep{anthropic2023claude, openai2023gpt}, information retrieval and AI-assistant systems \citep{microsoftcopilot2026}, programming aids \citep{githubcopilot2026}, and even clinical decision support tools \citep{cascella2023evaluating}, demonstrating remarkable versatility in professional and consumer domains alike.

Efficient inference represents a fundamental challenge in large language model deployment, focusing on how models process input prompts and generate responses with minimal latency. The inference pipeline consists of two distinct phases. (1) \textbf{Prefill:} processing the input prompt (a user-submitted text sequence), and (2) \textbf{Decode:} autoregressively generating output tokens (where each token typically corresponds to a word or sub-word unit). Many widely used LLMs employ the Transformer architecture \citep{vaswani2017attention} for this process. For instance, when processing the prompt ``Why is the sea blue?", the model first tokenizes it into discrete units (``Why", ``is", ``the", ``sea", ``blue", ``?"), then sequentially generates output tokens (e.g., beginning with ``Because") while considering both the prompt and previously generated tokens at each step.
The core computational challenge emerges in multi-request scenarios, where the system must optimally schedule numerous concurrent inference tasks. This scheduling problem requires careful balancing of computational resources to minimize user-perceived latency, particularly given the autoregressive nature of token generation and the variable length of both input prompts and output responses.

The scheduling problem in LLM inference presents unique challenges that distinguish it from traditional scheduling tasks, primarily due to its dynamic memory constraints. During inference, each computational worker maintains a fixed Key-Value (KV) cache memory capacity, where keys and values represent contextual embeddings derived from both input tokens and autoregressively generated output tokens. The memory capacity depends on the hardware of each worker and the complexity of the LLM used. This memory system exhibits two critical behaviors: (1) it must retain all processed tokens from the prefill input prompt, and (2) its consumption grows linearly during decode generation as each newly generated output token requires additional KV cache allocation. The memory growth pattern creates complex scheduling dynamics when combined with parallel processing requirements. Unlike traditional systems where task resource demands remain static, LLM inference exhibits variable memory pressure - the number of concurrent requests a worker can handle fluctuates dynamically based on each request's generation progress. Although each request's memory footprint grows linearly with its sequence length, the resulting interaction between memory availability and concurrency complicates scheduling. The challenge is further compounded by the need to balance immediate memory availability against anticipated future demands as generation progresses across multiple concurrent requests.

\cite{jaillet2025online} established a theoretical model for LLM inference scheduling on a single computational worker, proposing a shortest-first scheduling policy that prioritizes requests with fewer output tokens. Their approach dynamically batches requests while respecting memory constraints, demonstrating both theoretical and numerical effectiveness under the condition of uniform input sizes. However, this condition may not always hold in practice, where workers typically handle mixed workloads (e.g., combining short conversational prompts with lengthy document summarization tasks).

Relaxing this condition fundamentally changes the scheduling problem. With variable prefill input sizes, the relationship between processing time, memory usage, and optimal scheduling becomes significantly more complex. For instance, a request with a large input but small output may consume substantial memory yet finish quickly, while one with a small input but long output starts with a small memory footprint that grows throughout a long decoding phase. This trade-off makes prioritization non-trivial, as the previously effective shortest-first heuristic becomes inadequate. In Section~\ref{subsec:sf} and Appendix~\ref{append:model}, we show that the generalized scheduling problem is NP-hard and that prioritizing by either output length or total sequence length can yield an unbounded approximation ratio. Some other related work can be found in Appendix \ref{sec:other}.

\subsection{Main Contribution} \label{subsec:main}

In this section, we introduce the main contributions and the roadmap of this paper.

\xhdr{Negative Results for Existing Scheduling Algorithms and NP-hardness. }
Section~\ref{sec:model} and Appendix~\ref{append:model} establish two limitations of the heterogeneous-input model: shortest-output-first and shortest-total-size-first scheduling can have unbounded approximation ratios, and minimizing total end-to-end latency is NP-hard. These results formally justify the need for new scheduling approaches.

\xhdr{Scheduling with a Constant-Factor Approximation Guarantee. }
In Section~\ref{sec:algorithm}, we introduce \sortF{}, which repeatedly selects a feasible batch minimizing $F(\mathcal X)$ and orders its requests by nondecreasing output length before executing the resulting list through continuous admission. With exact batch selection at every Phase~1 iteration, Section~\ref{sec:proof} establishes a factor-$40$ approximation for offline total latency under known output lengths and uninterrupted unit-time decoding. The guarantee holds for every finite feasible instance, with all requests available at time zero, under either synchronous dynamic batch feasibility or a static peak-memory constraint. Both versions are compared with the optimum of the original dynamic-memory model. Our analysis combines a comparison with the serial batch order, a bound on the cost of batching requests completed by a deadline, and a residual-set lower bound on total completion time.

\xhdr{Implementations for Practical Deployment.}
Phase~1 batch selection is the main computational bottleneck of \sortF{}. In Section~\ref{sec:approximation}, we develop a capacity-indexed dynamic program that minimizes $F$ exactly under a static peak-memory constraint, together with Local Swap Search and Quantile Greedy Selection for larger instances. The dynamic program is pseudopolynomial in the memory capacity $M$ and preserves the factor-$40$ guarantee when used at every Phase~1 iteration. The two heuristics reduce computational overhead and are evaluated empirically; their worst-case approximation guarantees remain open. These implementations offer different trade-offs between selection quality and computational cost. We also discuss using output-length predictions and adaptive refinement in practical scheduling.

\xhdr{LP-based Heuristics. }
Section~\ref{sec:extension} presents a time-indexed integer program based on \cite{jaillet2025online} and its linear programming relaxation. We establish a sufficient completion horizon for the formulation, so that the LP provides a valid lower bound in the unit-time decoding model. The fractional solution guides request ordering in \sortLP{}, while \LPSwap{} combines this ordering with F-metric-guided local refinement. Both heuristics use the feasible admission rule of \sortF{}. We do not characterize the integrality gap or claim approximation guarantees for these LP-guided schedules.

\xhdr{Numerical Experiments. }
In Section \ref{sec:num}, we present comprehensive empirical evaluations using real-world data to assess the practical performance of the scheduling methods. To capture heterogeneous request lengths, we use both the LMSYS-Chat dataset \citep{zheng2023lmsys} and the Arxiv long-document summarization dataset \citep{cohan2018discourse}, and evaluate performance across both unmixed and mixed workload compositions. We compare practical variants of \sortF{} against standard baselines including first-come-first-served, shortest-first, and LP-guided heuristics. The experiments report mean and standard deviation over multiple random seeds, sensitivity to workload composition, tail latency, and scheduling runtime. We further include a preliminary online experiment under Poisson arrivals to evaluate a receding-horizon version of \sortF{}. The results show substantial latency reductions for \sortF{} with Local Swap Search on mixed workloads and indicate benefits from the receding-horizon policy when an online backlog accumulates.

\section{Model Review} \label{sec:model}

In this section, we review the scheduling model introduced by \cite{jaillet2025online} for LLM inference under memory constraints. Before presenting the formal model, we provide background on two key features of modern LLM serving systems—\emph{prefill/decode mixing} and \emph{continuous batching}—that the model captures. We then discuss our extension, which relaxes a strong assumption on input sizes made in the original work.

\paragraph{Prefill/Decode (PD) mixing.} 
Modern LLM serving systems support prefill/decode (PD) mixing \cite{kwon2023efficient}, whereby a single processing batch can simultaneously contain requests in different phases—some undergoing prefill while others are mid-decode. This mixing improves hardware utilization by combining diverse workloads within each batch.

\paragraph{Continuous Batching.} 
Traditional batching forms a fixed set of requests and processes them until all complete before admitting new requests. In contrast, continuous batching dynamically admits new requests into an ongoing batch as soon as memory capacity becomes available—typically when an in-progress request completes and releases its KV cache. This approach significantly improves throughput by avoiding idle capacity while waiting for long-running requests to finish.

Figure~\ref{fig:worker} illustrates the scheduling process. Circles represent prefill operations and squares represent decode operations, with colors distinguishing requests. Operations aligned vertically form a batch processed in parallel. The figure demonstrates PD mixing (e.g., prefill and decode operations co-occur in the same batch) and continuous batching (new requests enter as earlier ones complete and release memory).

\begin{figure}[!h]
\centering
\includegraphics[width=0.8\textwidth]{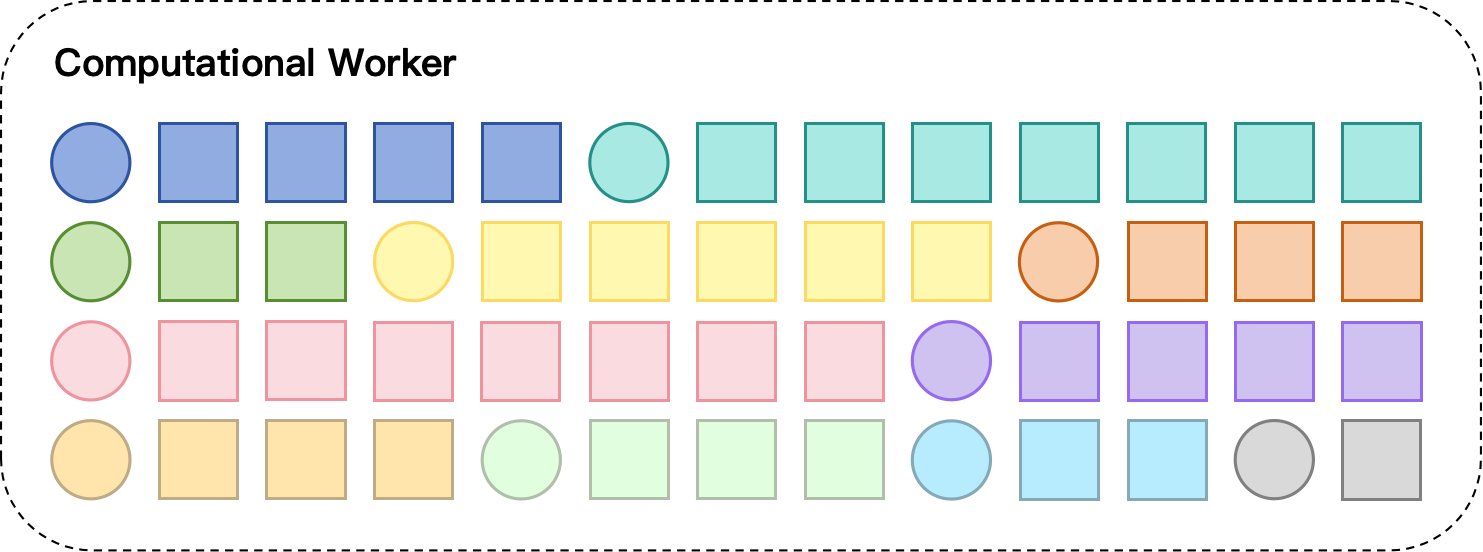}
\caption{Illustration of LLM inference scheduling with prefill/decode mixing and continuous batching. Circles denote prefill operations; squares denote decode tokens; colors distinguish requests. Operations in the same column constitute a batch processed simultaneously. The schedule exhibits PD mixing (batches contain both prefill and decode operations) and continuous batching (requests enter dynamically as memory becomes available).}
\label{fig:worker}
\end{figure}

For readability, Table~\ref{tab:notation} in Appendix \ref{append:notation} summarizes the main notation used throughout the paper.

\subsection{Mathematical Model Review}

\cite{jaillet2025online} introduce a discrete-time system with a single computational worker. The worker has a KV cache capacity of \(M > 0\) tokens; in practice, \(M\) depends on model size and hardware specifications (e.g., GPU memory) and is assumed known to the scheduler.

\paragraph{Request Arrivals and Adversarial Setting.}
We consider \(n\) prompts that arrive simultaneously at time \(t = 0\). Each prompt \(i \in [n]\) is characterized by a pair \((s_i, o_i)\), where \(s_i\) denotes the prefill size (number of input tokens) and \(o_i\) denotes the number of output tokens to be generated. We adopt an adversarial framework: the adversary selects the number of requests \(n\) and the parameters \((s_i, o_i)\) for each request, subject to the individual feasibility condition stated below.

\paragraph{Time Convention and Individual Feasibility.}
All lengths and initiation times are integers, with $s_i,o_i\ge1$, $p_i\ge0$, and $s_i+o_i\le M$ for every request. A request initiated at boundary $p_i$ generates token $j\in[o_i]$ in slot $p_i+j$ and completes at $c_i=p_i+o_i$. Decoding is uninterrupted once a request starts. The analytical model accounts for prefill through the initial KV-cache footprint $s_i$ and assigns one unit of time to each decoding slot; it does not assign a separate duration to prefill computation.

The approximation guarantee holds for every finite feasible instance. In particular, it does not require $L(M)/M\to0$, where $L(M):=\max_i(s_i+o_i)$, or any lower bound on the workload size. We take $M$ to be an integer token capacity; a noninteger capacity is equivalent to its floor.

\paragraph{Output Length Information.}
In the theoretical analysis, we assume that output lengths \(o_i\) are known to the scheduler. This assumption isolates the fundamental scheduling challenge, which is already NP-hard (see Section \ref{subsec:sf}). In practice, \(o_i\) can be predicted, and several effective prediction methods exist \citep{zheng2023response,qiu2024efficient,shahout2024don}. However, none of these prediction methods can guarantee accuracy. In Section~\ref{sec:approximation}, we discuss using predictions and adaptive length refinement as practical extensions.

\paragraph{Batch Processing and Memory Constraint.} Each request \(i\) must be processed token by token during the decode phase. The memory required to generate the \(j\)-th output token (\(j \in [o_i]\)) is \(s_i + j\), reflecting the growing KV cache. At each time step \(t\), the system processes a batch \(\mathcal{B}_t\)  that may contain tokens from multiple requests at various stages of completion. Consistent with sequential decoding, at most one token from each request can appear in any batch. The model captures continuous admission and the growing KV-cache footprint. The PD-mixing illustration provides serving-system context; the approximation analysis uses the unit-time decoding abstraction specified above. The memory constraint requires
\[
\sum_{i \in \mathcal{B}_t} (s_i + d_i(t)) \leq M,
\]
where $d_i(t):=t-p_i$ is the decoding age (equivalently, the output-token index) of request $i$ in slot $t$, defined for $i\in\mathcal B_t=\{i:p_i<t\le p_i+o_i\}$. Thus $d_i(t)\in\{1,\ldots,o_i\}$ for every active request. We reserve $a_i$ for the arrival time of request $i$ in the online setting; $a_i=0$ for all requests in the offline model.

\paragraph{Backlogged Workloads and the Offline Setting.}
The offline model represents a finite backlog available at time zero. High demand motivates this setting, but overload is not an assumption of our theorem.

The simultaneous-arrival assumption is intended to model burst or backlogged serving regimes that arise during peak-demand periods. Such regimes are consistent with recent empirical studies of LLM serving workloads, which report substantial burstiness and temporal variation in request arrivals, including periodic and aperiodic workload patterns over minute-, hour-, and day-level timescales \citep{wang2024burstgpt}. Related LLM-serving studies also evaluate systems under diurnal or temporally varying load patterns to model production request-rate variation \citep{goel2026qoserve, jiang2026oserve}. During low-load periods, requests can often be admitted shortly after arrival, and the scheduling problem is less consequential. During peak or transient-overload periods, however, arrivals may exceed service capacity and create a backlog of pending requests that must be prioritized under the KV-cache memory budget. Our assumption that all \(n\) prompts arrive simultaneously at \(t=0\) abstracts this backlogged state and isolates the scheduling difficulty in this high-congestion regime.

We emphasize that the approximation guarantee developed in this paper applies to the offline/backlogged model. To provide preliminary evidence beyond this setting, Section~\ref{sec:num} also evaluates a receding-horizon version of our offline algorithm under Poisson arrivals.

\subsection{Objective and Performance Metric}

The objective is to minimize \emph{total end-to-end latency} (TEL), defined as the sum of completion times. Since all requests arrive at \(t = 0\), the latency of request \(i\) equals its completion time \(c_i(\Lambda)\) under schedule \(\Lambda\):
\[
\text{TEL}(\Lambda; \mathcal{I}) = \sum_{i \in [n]} c_i(\Lambda).
\]

\cite{jaillet2025online} formulate an integer linear program (ILP) to compute optimal schedules. However, solving the ILP is computationally prohibitive for real-time deployment, where decisions must be made within milliseconds. Our goal is therefore to design efficient approximation algorithms with provable performance guarantees.

Let \opt{} denote an optimal schedule for request set \(\mathcal{I}\). The performance of algorithm \(\Lambda\) is measured by its \emph{approximation ratio}:
\[
\text{AR}(\Lambda) = \sup_{\mathcal{I}\ne\varnothing} \frac{\text{TEL}(\Lambda; \mathcal{I})}{\text{TEL}(\text{\opt}; \mathcal{I})}.
\]
The supremum ranges over nonempty finite feasible instances. The approximation ratio is always at least 1, with smaller values indicating better worst-case performance.

\subsection{Challenges with Heterogeneous Input Sizes} \label{subsec:sf}

\cite{jaillet2025online} propose a \emph{memory-constrained shortest-first} algorithm (\mcsdf) that prioritizes requests by output length \(o_i\). Before adding a request to the current batch, the algorithm verifies that memory constraints will remain satisfied at all future time steps.

At boundary $t$, let $\mathcal R^{(t)}$ be the pending requests before the admission decision and let $\mathcal V^{(t)}$ contain previously initiated requests that have not completed. For a nonempty candidate set $U\subseteq\mathcal R^{(t)}$ to be initiated at $t$, define $t_{\max}(U):=t+\max_{i\in U}o_i$. Its admission is feasible if
\begin{equation}
\label{eqn:Constraint}
\begin{aligned}
&\sum_{i\in\mathcal V^{(t)}}(s_i+t'-p_i)
    \mathbf 1_{\{0<t'-p_i\le o_i\}}
 +\sum_{i\in U}(s_i+t'-t)
    \mathbf 1_{\{0<t'-t\le o_i\}}\le M,\\[-1mm]
&\hfill t'=t+1,\ldots,t_{\max}(U).
\end{aligned}
\end{equation}
After $t_{\max}(U)$, the new requests have completed and the previously feasible memory profile is unchanged. The rule scans requests in nondecreasing output length and admits the longest feasible prefix, stopping at the first request that fails this check.

\cite{jaillet2025online} establish strong theoretical and empirical performance of \mcsdf{} under the assumption that \(s_i = s\) for all \(i \in [n]\). With uniform prefill sizes, shortest-first scheduling is natural: requests with smaller \(o_i\) have lower peak memory usage (\(s + o_i\)) and shorter processing times, so prioritizing them improves both throughput and memory utilization.

However, uniform prefill sizes may not hold when systems handle diverse workloads mixing long and short prompts. Relaxing this assumption introduces fundamental difficulties. Consider requests \(i\) and \(j\) with \(s_i < s_j\) but \(o_i > o_j\): at the same decoding age, request $i$ has a smaller footprint but requires more decoding slots. The optimal priority depends on global interactions with other requests, making the trade-off nontrivial.

We show that \mcsdf{} can perform arbitrarily poorly when input sizes vary.

\begin{theorem} \label{thm:benchmark}
There is a family of finite feasible instances $\mathcal I_M$, indexed by an unbounded sequence of capacities $M$, such that
\[
\frac{\mathrm{TEL}(\textup{\mcsdf};\mathcal I_M)}
     {\mathrm{TEL}(\textup{\opt};\mathcal I_M)}\longrightarrow\infty
\quad\text{as }M\to\infty.
\]
\end{theorem}

The proof appears in Appendix~\ref{append:model}, which also shows that \mcsdff{}, the variant that prioritizes requests by total size \(s_i + o_i\) while retaining the same feasibility check and head-of-line admission rule, has an unbounded approximation ratio. Moreover, the problem becomes computationally hard without uniform inputs.\footnote{Whether minimizing TEL is NP-hard under uniform \(s_i\) remains open.}

\begin{theorem} \label{thm:latencynp}
    Under the model of \cite{jaillet2025online} with \(s_i \in [1,M]\), \(o_i \in [1,M]\), and \(s_i + o_i \leq M\), minimizing total end-to-end latency is NP-hard.
\end{theorem}

We prove Theorem~\ref{thm:latencynp} via a reduction from the restricted 3-Partition problem in Appendix~\ref{append:model}. In this restricted version, the input integers \(x_i\) satisfy \(T/4 < x_i < T/2\), and the problem remains strongly NP-hard; see \cite{gareyjohnson1979}, Problem~SP15.

Together with the constant-factor approximation guarantee established later for \sortF{}, Theorem~\ref{thm:latencynp} places the present problem in the same broad complexity landscape as many classical batch scheduling problems with resource constraints: exact optimization is computationally hard, yet an exact batch-selection oracle can yield a constant performance guarantee. Our static implementation is pseudopolynomial in the token capacity $M$. The key distinction in our model is that the resource consumption is induced by the KV cache and evolves during autoregressive decoding, so the priority rule must account for both batch throughput and memory feasibility over time.

\section{Algorithm Description} \label{sec:algorithm}

This section presents our novel scheduling algorithm designed to minimize total end-to-end latency. The algorithm dynamically schedules requests by balancing memory constraints with a quality metric that optimizes both batch concurrency and response length efficiency.

At the core of our approach is the quality metric $F(\mathcal{X})$ for any nonempty request set $\mathcal{X}$, defined as
$$
F(\mathcal{X}) = \frac{\sum_{r_i \in \mathcal{X}} o_i}{|\mathcal{X}|^2},
$$
where $\mathcal{X}$ is a set of requests, and $|\mathcal{X}|$ is the cardinality of the request set.
Smaller values of $F(\mathcal{X})$ indicate higher scheduling priority. The metric balances batch size against response length, as the example below illustrates.

The metric also admits a natural interpretation from the perspective of classical scheduling theory. Let
\[
\bar{o}(\mathcal{X}) :=
\frac{\sum_{r_i\in \mathcal{X}} o_i}{|\mathcal{X}|}
\]
denote the average response length of requests in batch $\mathcal{X}$. Then
\[
F(\mathcal{X})
=
\frac{\bar{o}(\mathcal{X})}{|\mathcal{X}|},
\qquad
\frac{1}{F(\mathcal{X})}
=
\frac{|\mathcal{X}|}{\bar{o}(\mathcal{X})}.
\]
The reciprocal $1/F(\mathcal X)$ measures batch cardinality per unit of average output length. This resembles classical processing-time-per-weight priorities, including Smith's rule for $1||\sum w_jC_j$. It is an analogy rather than an optimal sequencing rule for dynamically overlapping batches: a synchronous batch lasts $D(\mathcal X):=\max_{i\in\mathcal X}o_i$, not its average output length. For an exactly selected batch, Lemma~\ref{lemma:output_balance} gives
\[
\frac{|\mathcal X|}{D(\mathcal X)}\le\frac1{F(\mathcal X)}
<\frac{2|\mathcal X|}{D(\mathcal X)},
\]
so the reciprocal metric approximates the batch's completion rate within a factor of two. Related batch scheduling models are discussed in \citep{brucker1998scheduling,chen2008logistics}.

To illustrate its operational intuition, we begin with a concrete numerical example that demonstrates how \sortF{} outperforms shortest-first scheduling (\mcsdf):

% ===== NUMERICAL EXAMPLE FIRST =====
\begin{example} \label{exp1}
Consider a scenario with KV cache memory $M = 64$ and two request types:
\renewcommand{\arraystretch}{1.2}
\begin{table}[htbp]
\centering
\caption{Request Types for Memory-Constrained Scheduling}
\label{tab:hetero_example}
\begin{tabular}{|l|l|l|l|}
\hline
\textbf{Type} & \textbf{Prompt Size} & \textbf{Response Length} & \textbf{Requests Number} \\ \hline
1 & 63 & 1 & 1 \\ \hline
2 & 1 & 2 & 21 \\ \hline
\end{tabular}
\end{table}

Under \mcsdf{}, which prioritizes shorter outputs, Type 1 would be scheduled first since $o_1 = 1 < o_2 = 2$. After Type 1 completes at time $t=1$, Type 2 requests are processed in batches. With memory constraint $M=64$, each batch can accommodate $\lfloor 64/(1+2) \rfloor = 21$ Type 2 requests simultaneously. The total end-to-end latency (TEL) is
$$
\text{TEL}(\text{\mcsdf}\,;\mathcal{I}) = \underbrace{1}_{\text{Type 1}} + \underbrace{21 \times 3}_{\text{Type 2}} = 64.
$$

In contrast, \sortF{} computes the F-metric for each type:
\begin{align*}
F(\text{Type 1}) &= \frac{1}{1^2} = 1 \\
F(\text{Type 2}) &= \frac{2 \times 21}{21^2} = \frac{42}{441} \approx 0.095.
\end{align*}
Since $F(\text{Type 2}) < F(\text{Type 1})$, \sortF{} prioritizes Type 2 first. Type 2 completes at $t=2$, after which Type 1 is processed at $t=3$. The TEL is
$$
\text{TEL}(\text{\sortF}\,;\mathcal{I})= \underbrace{21 \times 2}_{\text{Type 2}} + \underbrace{3}_{\text{Type 1}} = 45.
$$
This represents a $29.7\%$ reduction in TEL compared to \mcsdf, demonstrating how \sortF{}'s F-metric effectively balances batch size and processing time to minimize latency.
\end{example}

Building on this concrete demonstration, we generalize the intuition through a symbolic example (Example \ref{exp2}) in Appendix \ref{append:example} that shows how the F-metric selects the better of two fixed serial batch orders.

Examples \ref{exp1} and \ref{exp2} illustrate how the F-metric balances batch size and response lengths in scheduling decisions.

A key challenge arises from variability in output sizes: when requests are batched together, they do not necessarily finish processing at the same time. Under our model, once a request completes, its memory resources are freed, allowing new requests to be dynamically added. Large differences in output lengths can leave a small number of requests active long after the others finish.

To address this, one useful property---formalized in Lemma \ref{lemma:output_balance}---ensures that within any exactly selected batch, no single request's output length dominates the average. The proof can be found in Appendix \ref{append:alg}. This balancing characteristic helps avoid stragglers and also provides analytical leverage in Section~\ref{sec:proof}.

\begin{lemma} \label{lemma:output_balance}
For any nonempty batch $\mathcal{X}$ selected by exact Phase~1 minimization, let $b=|\mathcal{X}|$, $o_{\max}=\max_{r_i\in\mathcal{X}}o_i$, and $\bar{o}=b^{-1}\sum_{r_i\in\mathcal{X}}o_i$. Then
\begin{equation}
    \bar{o}\ge \frac{b}{2b-1}\,o_{\max}>\frac12o_{\max}.
\end{equation}
The statement holds for both the dynamic synchronous-feasibility rule and the static peak-memory rule used for exact Phase~1 selection.
\end{lemma}

After introducing the intuition and a useful property behind the F-metric, we now describe our scheduling algorithm, \sortF. Let $\mathcal{I}$ denote the set of all requests in the instance. The algorithm has two phases. Phase~1 constructs an ordered list $\mathcal{I}'$ by repeatedly selecting a nonempty feasible batch with the smallest $F(\cdot)$, with ties broken by preferring a larger batch. A dynamically feasible Phase~1 batch is one whose requests can be initiated simultaneously; equivalently,
\begin{equation}
    \sum_{r_i\in\mathcal{X}:\,o_i\ge u}(s_i+u)\le M,
    \qquad u=1,\ldots,\max_{r_i\in\mathcal{X}}o_i.
\end{equation}
We also consider the static candidate family $\sum_{r_i\in\mathcal{X}}(s_i+o_i)\le M$, which is sufficient for dynamic feasibility and is the feasible family solved exactly by the DP in Section~\ref{sec:approximation}.

In Phase~1, selected requests within each batch are sorted by nondecreasing $o_i$, and the ordered batches are concatenated into $\mathcal I'$. This within-batch order is part of the algorithm and is used in the proof.

Phase~2 executes this list at boundaries $t=0,1,\ldots$. After removing requests completed by $t$, it scans the pending requests in list order and assigns $p_i=t$ to each newly admitted request. It admits the longest prefix whose joint future memory profile with the active requests satisfies \eqref{eqn:Constraint}, stopping at the first infeasible request. All admitted, unfinished requests then generate one token in slot $t+1$.

For a set $\mathcal X$ with assigned initiation times, write
\begin{equation}
\label{eq:sf_memory_profile}
M(\mathcal X,u)=\sum_{r_i\in\mathcal X}(s_i+u-p_i)
\mathbf 1_{\{0<u-p_i\le o_i\}}.
\end{equation}
Because the footprint increases between completion events, admission can be checked at the future completion slots of active and tentatively admitted requests. Algorithm~\ref{algorithm_1} in Appendix~\ref{append:alg} gives the complete procedure. Phase~2 always uses this dynamic memory profile, including when Phase~1 uses the static candidate family.

\section{Constant-Factor Approximation Guarantee}
\label{sec:proof}

We analyze \sortF{} with exact Phase~1 batch selection at every iteration and the ordered
admission rule in Algorithm~\ref{algorithm_1}. Within each selected batch,
requests are ordered by nondecreasing output length. Phase~2 admits requests
in the resulting list order and stops at the first request that cannot be
admitted without violating a future memory constraint. Throughout the
analysis, a request initiated at time $p_i$ generates its tokens in slots
$p_i+1,\ldots,p_i+o_i$ and completes at time $p_i+o_i$.

\begin{theorem}
\label{thm:cr}
For every finite instance $\mathcal I$ in the unit-time, uninterrupted-decoding
model, with known positive integer lengths satisfying $s_i+o_i\le M$ for all
$r_i\in\mathcal I$,
\sortF{} satisfies
\[
    \mathrm{TEL}(\textup{\sortF};\mathcal I)
    \le 40\,\mathrm{TEL}(\textup{\opt};\mathcal I).
\]
Consequently, $\mathrm{AR}(\textup{\sortF})\le40$.
The same guarantee holds if Phase~1 instead minimizes $F$ exactly over
batches satisfying $\sum_{r_i\in\mathcal X}(s_i+o_i)\le M$, with Phase~2
unchanged.
\end{theorem}

The guarantee compares both versions with the optimum of the original
dynamic-memory model. It requires neither a small-request scaling assumption
nor an overload condition. The proof bounds the algorithm's completion
times using its batch order, then derives a lower bound on the optimum from
the number of requests that any feasible schedule can complete by a given
time. In particular, the two schedules need not use the same batch partition.

For a nonempty set $\mathcal B$, write $D(\mathcal B):=\max_{r_i\in\mathcal B}o_i$. Fix a nonempty instance, and let $\mathcal X_1,\ldots,\mathcal X_K$ be the
batches selected in Phase~1. Define
\[
    n_k=|\mathcal X_k|,\qquad
    D_k=\max_{r_i\in\mathcal X_k}o_i,\qquad
    f_k=F(\mathcal X_k),
\]
\[
    \mathcal R_k=\bigcup_{h=k}^K\mathcal X_h,\qquad
    N_k=|\mathcal R_k|,\qquad N_{K+1}=0,
    \qquad H=\sum_{k=1}^K N_kn_kf_k.
\]
Thus $\mathcal R_k$ is the residual set when batch $k$ is selected, and
$N_k-N_{k+1}=n_k$. The scalar $H$ connects the algorithmic upper bound
and the benchmark lower bound below. All three lemmas apply to either
Phase~1 selection rule in Theorem~\ref{thm:cr}; their proofs are given in
Appendix~\ref{append:proof}.

\begin{lemma}[Serial comparison]
\label{lemma:sf_serial}
Let $T_k=\sum_{h<k}D_h$. Every request in $\mathcal X_k$ is initiated by
$T_k$ under \sortF{}. Moreover,
\begin{equation}
\label{eq:sf_algorithm_bound}
    \mathrm{TEL}(\textup{\sortF};\mathcal I)
    \le\sum_{k=1}^K N_kD_k\le2H.
\end{equation}
\end{lemma}

The comparison accounts for all earlier batches that may still be active.
Its key property is that a synchronously feasible batch remains feasible
when its requests are initiated at nondecreasing times in nondecreasing
output-length order. Lemma~\ref{lemma:output_balance} then bounds each batch's
maximum output length by twice its average.

\begin{lemma}[Batching requests completed by a deadline]
\label{lemma:sf_deadline}
Let $\mathcal S$ be any subset of requests completed by time $t>0$ in a
feasible schedule. There exists a partition
$\mathcal S=\mathcal B_1\cup\cdots\cup\mathcal B_m$ into nonempty batches such
that
\begin{equation}
\label{eq:sf_deadline_bound}
    \sum_{r_i\in\mathcal B_j}(s_i+o_i)\le M
    \quad\text{for every }j,
    \qquad
    \sum_{j=1}^m\max_{r_i\in\mathcal B_j}o_i\le5t.
\end{equation}
For $\mathcal S=\varnothing$, the empty partition satisfies the same bound.
\end{lemma}

This lemma converts a completed subset into static-feasible batches without
changing any request length. It is applied separately at each deadline and
does not require a single transformed optimal schedule.

\begin{lemma}[Residual-set lower bound]
\label{lemma:sf_lower_bound}
Fix an optimal schedule with completion times $c_i^*$, and set
\[
    q_k(t)=|\{r_i\in\mathcal R_k:c_i^*\le t\}|.
\]
Then, for every $k$ and $t\ge0$,
\begin{equation}
\label{eq:sf_residual_bound}
    q_k(t)f_k\le5t.
\end{equation}
Furthermore,
\begin{equation}
\label{eq:sf_optimal_bound}
    \mathrm{TEL}(\textup{\opt};\mathcal I)\ge\frac{H}{20}.
\end{equation}
\end{lemma}

To obtain \eqref{eq:sf_residual_bound}, apply
Lemma~\ref{lemma:sf_deadline} to the requests counted by $q_k(t)$. Each batch
in the resulting partition is a feasible candidate when $\mathcal X_k$ is
selected. Exact $F$-minimization gives
$f_k|\mathcal B_j|\le D(\mathcal B_j)$; summing over these batches proves
\eqref{eq:sf_residual_bound}. Integrating the resulting lower bounds on the number of unfinished
requests gives \eqref{eq:sf_optimal_bound}. The last step adapts the histogram
charging argument of \citet{feige2004approximating} to the present
dynamic-memory setting.

\begin{proof}{Proof of Theorem~\ref{thm:cr}.}
For a nonempty instance, Lemmas~\ref{lemma:sf_serial}
and~\ref{lemma:sf_lower_bound} give
\[
    \mathrm{TEL}(\textup{\sortF};\mathcal I)
    \le2H
    \le40\,\mathrm{TEL}(\textup{\opt};\mathcal I).
\]
The empty instance is immediate. Both Phase~1 selection rules satisfy the
conditions used in the lemmas, so the conclusion holds for both versions.
\Halmos
\end{proof}

\section{Batch Selection and Practical Implementation under Prediction Uncertainty} \label{sec:approximation}

Phase~1 of \sortF{} repeatedly minimizes $F$ over nonempty feasible subsets
of the current residual set. Exact selection over either candidate family in
Section~\ref{sec:algorithm} gives the factor-$40$ guarantee. We provide an
exact dynamic program for the static family and two heuristics for larger
workloads.

The static peak-memory condition
\[
    \sum_{i\in\mathcal X}(s_i+o_i)\le M
\]
is a conservative inner approximation of synchronous dynamic feasibility.
Indeed, for every decoding age $u$,
\[
    \sum_{i\in\mathcal X:o_i\ge u}(s_i+u)
    \le
    \sum_{i\in\mathcal X:o_i\ge u}(s_i+o_i)
    \le
    \sum_{i\in\mathcal X}(s_i+o_i)
    \le M.
\]
Thus, every statically feasible batch is also synchronously dynamically
feasible. The static condition charges each request its peak footprint
$s_i+o_i$, replacing the age-dependent family of constraints in~(3) by a
single additive capacity constraint. This simplification makes the Phase~1
selection problem knapsack-like and enables the exact capacity-indexed
dynamic program below.

The tradeoff is reduced batching flexibility: the static family can exclude
dynamically feasible batches because requests with different decode lengths
need not attain their peak footprints simultaneously. Nevertheless, when
$F$ is minimized exactly over this static family and Phase~2 retains the
original dynamic admission rule, Theorem~3 still compares the resulting
schedule with the optimum of the original dynamic-memory model and gives the
same factor-$40$ guarantee.

\subsection{Exact Dynamic Programming under a Static Memory Constraint}

Fix a residual set $\mathcal R$ of $r$ requests and write $w_i=s_i+o_i$. The static subproblem minimizes $F(\mathcal X)$ subject to $\varnothing\ne\mathcal X\subseteq\mathcal R$ and $\sum_{i\in\mathcal X}w_i\le M$. After processing any prefix of the residual requests, define
\[
dp[k][m]=\min\left\{\sum_{i\in\mathcal X}o_i:
\mathcal X\text{ is in the processed prefix},\ |\mathcal X|=k,
\ \sum_{i\in\mathcal X}w_i=m\right\}.
\]
An unattainable state has value $+\infty$; initially $dp[0][0]=0$ and all other states are unattainable. Processing request $i$ gives the update
\[
dp[k+1][m+w_i]\leftarrow
\min\{dp[k+1][m+w_i],\ dp[k][m]+o_i\},\qquad m+w_i\le M.
\]
The cardinality index is traversed in decreasing order, so no request is used more than once. After all requests are processed, select a finite state
\[
(k^*,m^*)\in\arg\min_{\substack{1\le k\le r\\0\le m\le M}}
\left(\frac{dp[k][m]}{k^2},-k\right),
\]
where the pair is minimized lexicographically, and return the request set reconstructed from that state. This is exact static batch selection, including the larger-cardinality tie rule.

The recurrence requires $O(r^2M)$ time and $O(rM)$ scalar storage. Algorithm~\ref{alg:exact_dp} in Appendix~\ref{append:approx} uses immutable predecessor records for reconstruction, retaining the $O(r^2M)$ time bound with $O(r^2M)$ total storage. Repeated selection has a coarse $O(n^3M)$ time bound over Phase~1. These bounds are pseudopolynomial in $M$; the DP is appropriate when both the residual set and the capacity state space are manageable.

\subsection{Local Swap Search}

This heuristic algorithm efficiently refines an initial solution through iterative local improvements. Starting with a feasible batch $\mathcal{X}_0$ constructed by sorting requests in ascending order of $s_i + o_i$ and applying memory-constrained greedy selection, the method systematically explores request swaps to progressively enhance solution quality. The approach exploits the combinatorial structure of the optimization landscape while avoiding exhaustive search.

The core operation is pairwise exchange: for any $r_{out} \in \mathcal{X}$ and $r_{in} \notin \mathcal{X}$ satisfying
\begin{enumerate}
    \item Memory feasibility: $\text{mem} + (s_{in} + o_{in}) - (s_{out} + o_{out}) \leq M$,
    \item Quality improvement: $F(\mathcal{X}') < F(\mathcal{X})$ with $\mathcal{X}' = (\mathcal{X} \setminus \{r_{out}\}) \cup \{r_{in}\}$.
\end{enumerate}

Formally, the iterative refinement is described by
$$
\mathcal{X}_{k+1} = 
\begin{cases} 
\mathcal{X}_k & \nexists \text{ improving swap} \\
\mathcal{X}_k \oplus (r_{in}, r_{out}) & F(\mathcal{X}_k \oplus (r_{in}, r_{out})) < F(\mathcal{X}_k),
\end{cases}
$$
where $\oplus$ denotes the exchange operation. The algorithm terminates at local minima where no improving swaps exist. The complete procedure is formalized in Algorithm \ref{alg:local_swap}, which can be found in Appendix \ref{append:approx}.

For a residual set of $r$ requests, sorting costs $O(r\log r)$ and each complete swap scan costs $O(r^2)$ when the memory total and output sum are maintained incrementally. With $P$ scans, the cost is $O(r\log r+Pr^2)$ per selected batch. Strict improvement guarantees termination; Appendix~\ref{append:approx} bounds $P$ in terms of $r$ and the selected cardinality. This method is a heuristic: it may get trapped in local minima and we do not provide a worst-case approximation guarantee for the Phase~1 objective.

\subsection{Quantile Greedy Selection}

This heuristic combines quantile filtering with capacity-aware greedy selection. For a residual set $\mathcal R$ of $r$ requests, sample $\max\{1,\lfloor r/2\rfloor\}$ requests without replacement and let $Q_p$ and $Q_o$ be the sample $0.3$-quantiles of $w_i=s_i+o_i$ and $o_i$, respectively. The core candidate set is
\[
\mathcal C=\{r_i\in\mathcal R:w_i\le Q_p,\ o_i\le Q_o\}.
\]
First scan $\mathcal C$ in nondecreasing $o_i$, adding a request only if its peak footprint fits in the remaining capacity. Then scan all unselected requests in nondecreasing $o_i/w_i$, again updating the used capacity after each admission. The selected batch is therefore static-feasible; the quantile filter alone does not define a feasible batch. Algorithm~\ref{alg:quantile_greedy} in Appendix~\ref{append:approx} gives the details.

The implementation takes $O(r\log r)$ time and $O(r)$ space per batch-selection call, including the two sorts. It is a lightweight heuristic with no established worst-case approximation guarantee.

\subsection{Algorithm Selection Guidelines}

The choice of solver depends on the residual queue size, the numeric capacity $M$, and the available scheduling time. Table~\ref{tab:algorithm_comparison} summarizes the cost of one batch-selection call; constructing the full order requires repeated calls.
\begin{table}[htbp]
\centering
Properties of the Phase~1 batch-selection solvers
\label{tab:algorithm_comparison}
\small
\begin{tabular}{p{0.25\textwidth}p{0.33\textwidth}p{0.32\textwidth}}
\hline
Solver & Time per call on $r$ requests & Scheduling guarantee\\
\hline
Exact static DP & $O(r^2M)$ & Factor $40$ with exact selection at every call\\
Local Swap Search & $O(r\log r+Pr^2)$; $P$ swap scans & No established worst-case bound\\
Quantile Greedy Selection & $O(r\log r)$ & No established worst-case bound\\
\hline
\end{tabular}
\end{table}

The DP supplies a provable guarantee when exact static selection is affordable. Local Swap Search offers local refinement, while Quantile Greedy Selection avoids repeated swap scans. Switching among solvers is a practical option, but the factor-$40$ theorem does not automatically apply to an order that includes heuristically selected batches.

\subsection{Handling Prediction Uncertainty via Adaptive Output Length Refinement} \label{subsec:prediction}

The algorithms presented thus far assume that output lengths $o_i$ are known exactly. In practice, output lengths must be predicted, and predictions are inherently imperfect. We discuss how the adaptive lower-bound strategy of \cite{chen2025adaptively} can guide a practical extension of \sortF{}.

\paragraph{Prediction Model.}
For request $i$, the predictor supplies positive integer bounds $[\ell_i,u_i]$. When comparing interval-based policies, we assume $1\le\ell_i\le o_i\le u_i\le M-s_i$; coverage is an assumption on the predictions, not a property guaranteed by the scheduling rule. The ratio $\gamma=\min_i\ell_i/u_i$ summarizes interval width.

\paragraph{Challenges with Na\"ive Approaches.}
A conservative policy, denoted by $\mathcal A_{\max}$, uses $u_i$ in the admission checks. With valid upper bounds and checks over the entire remaining decoding profile, this preserves memory feasibility but may reduce concurrency. Using lower bounds alone can understate future memory growth, so optimistic admission requires additional checks based on actual decoding progress.

\paragraph{Adaptive Refinement via $\mathcal A_{\min}$.}
The adaptive approach starts with $\tilde o_i=\ell_i$ and refines this estimate using observed progress. If request $i$ remains unfinished after generating $j$ output tokens, then its length is at least $j+1$, giving the update $\tilde o_i\leftarrow\max\{\tilde o_i,j+1\}$. Completion is determined by observed termination, rather than by reaching the current estimate.

\paragraph{Integration with \sortF.}
We denote the proposed combination by \sortF{}+$\mathcal A_{\min}$. Phase~1 uses the lower bounds to construct the initial priority order, and execution refines the length estimates as above. If increased memory requirements necessitate eviction, the proposed rule selects requests in nondecreasing $\tilde o_i$ and returns them to the pending queue. This eviction priority is heuristic: estimated total length does not determine remaining work. Any implementation must check actual memory requirements before each decoding step.

The handling of an evicted request's execution state and its associated cost are implementation dependent and lie outside the uninterrupted-decoding model. We therefore present this combination as a practical design direction; a complete execution model and performance guarantees remain to be established. Neither the factor-$40$ theorem nor guarantees for the original $\mathcal A_{\min}$ automatically transfer to this adaptation.

We do not evaluate \sortF{}+$\mathcal A_{\min}$ in the numerical experiments in Section~7. Accordingly, none of the reported numerical results assumes a particular state-retention/resume or restart semantics for an evicted request, and no recomputation, KV-cache migration, or state-restoration cost associated with eviction is included in those results. Evaluating such costs requires an explicit execution-state model and is left for future work.

\section{Extensions and Discussions} \label{sec:extension}

Building on the preceding analysis---where we proposed the \sortF{} algorithm and established its constant-factor approximation ratio for this NP-hard problem---we now discuss several complementary extensions. First, we use integer programming (IP) and linear programming (LP) to further illuminate the structure of the problem through relaxations and lower bounds, and to derive practical LP-guided scheduling heuristics. Second, motivated by streaming arrivals in production systems, we introduce a receding-horizon online variant of \sortF{} that repeatedly applies the F-metric to the currently available backlog.

\subsection{Heuristics from Integer Programming and Linear Programming Relaxation}

Adapting the time-indexed formulation of \cite{jaillet2025online} to our slot convention, the offline optimum can be written as the following integer program. Choose an integer completion horizon $\bar T\ge\sum_{i=1}^n o_i$. An optimal schedule fits within this horizon: idle slots can be deleted, and every occupied slot consumes at least one of the $\sum_i o_i$ output tokens. Appendix~\ref{append:extension_pseudo} gives a formal justification. Define the allowable start boundaries by
\[
    \mathcal T_i:=\{0,1,\ldots,\bar T-o_i\}.
\]
A variable $x_{i,t}=1$ means that request $i$ is initiated at boundary $t$ and therefore completes at time $t+o_i$. With this convention, the IP is
\begin{align}
\text{Optimal-IP} = \min\; & \sum_{i \in [n]}\sum_{t\in\mathcal T_i}(t+o_i)x_{i,t}  \label{eq:objective} \\
\text{s.t.}\, 
& \sum_{t\in\mathcal T_i}x_{i,t}=1, && \forall i \in [n] \label{eq:start_once} \\
& \sum_{i=1}^{n}\sum_{\substack{t\in\mathcal T_i\\ \max\{0,\tau-o_i\}\le t\le\tau-1}}(s_i+\tau-t)x_{i,t}\le M, && \forall \tau=1,\ldots,\bar T \label{eq:memory_constraint} \\
& x_{i,t}\in\{0,1\}, && \forall i\in[n],\ \forall t\in\mathcal T_i. \label{eq:integrality}
\end{align}
The memory constraint sums the footprints of exactly those requests that generate a token in slot $\tau$: a request started at boundary $t$ has decoding age $\tau-t$. Restricting $t$ to $\mathcal T_i$ makes $\bar T$ a completion horizon and ensures that every allowed start and every allowed completion is covered.

Since solving Optimal-IP is computationally intractable, we relax the integrality constraints and obtain the following LP:
\begin{align}
\text{Relaxed-LP} = \min\; & \sum_{i \in [n]}\sum_{t\in\mathcal T_i}(t+o_i)x_{i,t}   \\
\text{s.t.}\, 
& \sum_{t\in\mathcal T_i}x_{i,t}=1, && \forall i \in [n]  \\
& \sum_{i=1}^{n}\sum_{\substack{t\in\mathcal T_i\\ \max\{0,\tau-o_i\}\le t\le\tau-1}}(s_i+\tau-t)x_{i,t}\le M, && \forall \tau=1,\ldots,\bar T \label{eq:memory_constraint2} \\
& x_{i,t}\ge0. && \forall i\in[n],\ \forall t\in\mathcal T_i.
\end{align}

With this horizon and the same unit-time model, $\mathrm{Relaxed\text{-}LP}\le\mathrm{Optimal\text{-}IP}=\mathrm{TEL}(\textup{\opt};\mathcal I)$. A shorter horizon is not justified merely because it contains some feasible schedule; a lower bound for the horizon-restricted problem need not bound the unrestricted optimum. However, optimal LP solutions $x_{i,t}^{*}$ are typically fractional, and converting them into a feasible discrete schedule requires rounding or other heuristics. In what follows, we introduce two LP-guided scheduling heuristics, \sortLP{} and \LPSwap{}, whose effectiveness is assessed empirically (Appendix~\ref{append:simulation} and Section~\ref{sec:num}). We do not claim provable guarantees for these rounded schedules; in particular, the integrality gap of Relaxed-LP and the approximation quality of the rounding procedures are left for future work.

For priority construction alone, one may instead solve a shorter-horizon LP, provided it is feasible. The resulting ordering can still be executed feasibly by Phase~2, but that LP objective is not automatically a lower bound on the unrestricted optimum. All lower-bound comparisons in this paper require a justified horizon and a common timing model.

\paragraph{\sortLP: LP-induced ordering heuristic.}
The algorithm proceeds in three steps. First, it solves Relaxed-LP to obtain an optimal fractional solution $x_{i,t}^{*}$. Second, for each request $i$, it computes the LP-induced mean start boundary $y_i=\sum_{t\in\mathcal T_i}t x_{i,t}^*$. Finally, \sortLP{} sorts requests by increasing $y_i$ and then executes them using the same feasible admission rule as in Phase~2 of \sortF. Intuitively, smaller $y_i$ indicates that the LP relaxation ``prefers'' starting request $i$ earlier; we use this as an ordering signal rather than a guarantee. The detailed procedure is formalized in Algorithm~\ref{algorithm_sortLP} in Appendix~\ref{append:extension_pseudo}.

\paragraph{\LPSwap: LP ordering + F-metric refinement.}
While \sortLP{} provides a principled ordering signal from the LP relaxation, it may be further improved by incorporating our F-metric framework. To this end, we propose \LPSwap, a hybrid approach that starts from the \sortLP-induced ordering and then applies F-metric-guided local swap refinements when forming batches. This hybrid strategy leverages both the global perspective captured by the LP relaxation and the local optimization capability of the swap procedure. The complete pseudocode is provided in Algorithm~\ref{algorithm_LPswap} in Appendix~\ref{append:extension_pseudo}. We evaluate the effectiveness of these LP-based heuristics using simulated data in Appendix~\ref{append:simulation} and compare them on the real dataset in Section~\ref{sec:num}.

\subsection{Heuristics under Online Arrival}

\paragraph{\sortF{}-Online: receding-horizon F-metric scheduling.}
The offline model studied in the main theoretical analysis assumes that all requests are available at time $t=0$. To explore how the same scheduling principle can be adapted to streaming arrivals, we also consider a receding-horizon online variant, denoted by \sortF{}-Online. At each iteration $t$, the scheduler observes the pending backlog and uses an F-based priority order, refreshed at the times described below. The system then admits the longest feasible prefix of this order under the live residual KV-cache profile, exactly as in continuous batching. Hence, feasibility is checked against currently active requests at every iteration, while only the priority order over the pending backlog is updated by the online policy.

To reduce scheduling overhead, \sortF{}-Online uses a two-timescale implementation. On the slow timescale, it recomputes the F-based priority order every $10$ iterations, or earlier if the cached priority list is empty or if the backlog has changed substantially. In our implementation, this early refresh is triggered when the number of newly arrived pending requests reaches at least $50\%$ of the backlog size at the previous recomputation. On the fast timescale, at every iteration, the scheduler reuses the cached priority order, removes requests that have already started, appends newly arrived requests until the next recomputation, and admits the longest feasible prefix under the current memory profile.

The F-based planning step is also capped to keep online decisions lightweight. Specifically, \sortF{}-Online applies the local-swap version of F-metric batch construction only to a look-ahead window of at most $300$ pending requests. Once $300$ requests have been ordered, any remaining pending requests are appended using a simple fallback order based on increasing $s_i+o_i$. This cap prevents large online backlogs from causing excessive scheduling overhead, while still allowing the scheduler to optimize the front portion of the queue that is most likely to be admitted soon.

For arbitrary online arrivals at times $a_i$, the relevant objective is total flow time $\sum_i(c_i-a_i)$. Under uninterrupted decoding and without knowledge of future arrivals, \cite{jaillet2025online} establish the following limitation.

\begin{theorem}[Competitive Ratio Lower Bound \cite{jaillet2025online}]
    Under adversarial arrivals in the nonpreemptive model, no deterministic online algorithm has a competitive ratio bounded by a universal constant for total flow time, even when $s_i=s$ for all $i$.
\end{theorem}

Therefore, to evaluate the performance of this receding-horizon policy, we conduct numerical experiments under Poisson arrivals in Section~\ref{sec:num_online}.

\section{Numerical Experiments} \label{sec:num} 

In this section, we conduct numerical experiments with open-source real datasets to evaluate the performance of \sortF{} with heuristic Phase~1 solvers, alongside the LP-guided heuristics \sortLP{} and \LPSwap{}. We also include a preliminary online experiment under Poisson arrivals to assess a receding-horizon version of \sortF{}.

\xhdr{Dataset Overview. } To evaluate LLM inference performance under varying input prompt lengths—particularly in scenarios mixing short and long prompts—we combine two publicly available datasets, as no single existing dataset meets this need. The first dataset, introduced by \cite{zheng2023lmsys}, contains conversational data from over 210,000 unique IP addresses, collected via the Vicuna demo and Chatbot Arena. This dataset, available at \url{https://huggingface.co/datasets/lmsys/lmsys-chat-1m}, consists of short, interactive exchanges with a mean input length of 41.84 tokens (median: 12.00) and a mean output length of 85.35 tokens (median: 43.00), as shown in Figure \ref{fig:distribution1}. 

\begin{figure}[!h]
\center
\includegraphics[width=0.8\textwidth]{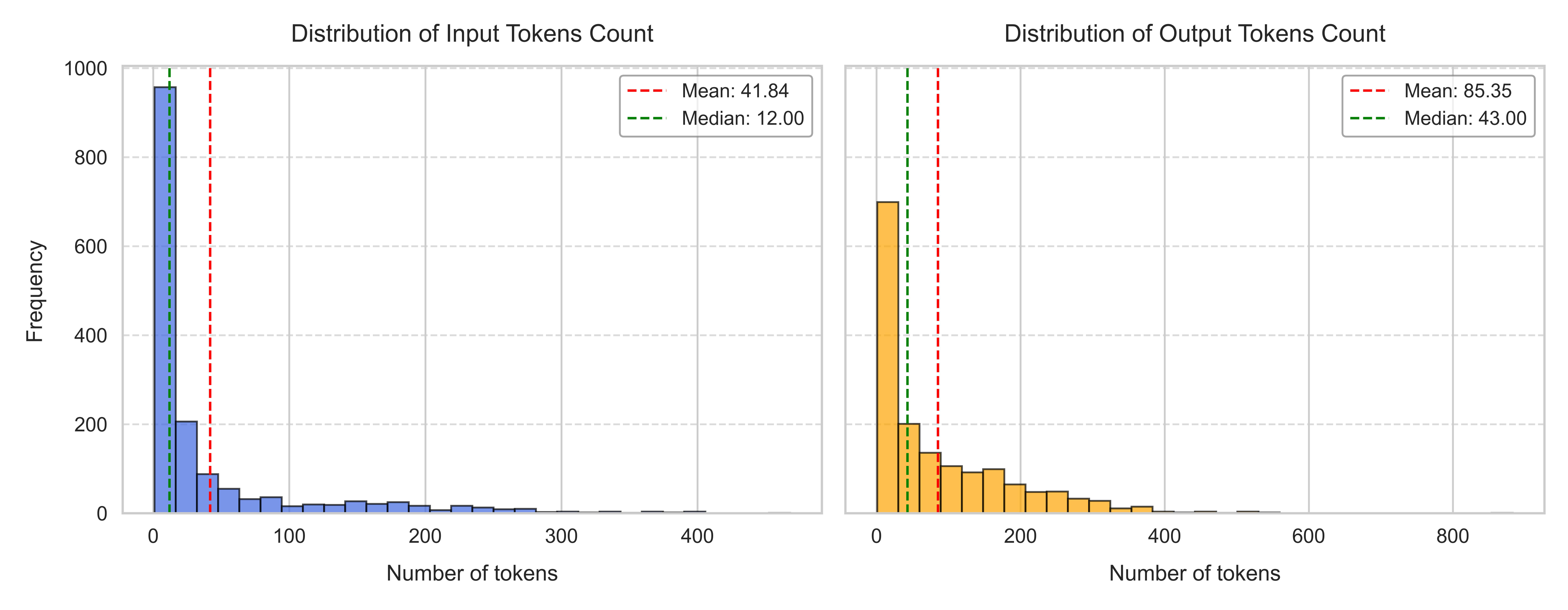}

\includegraphics[width=0.8\textwidth]{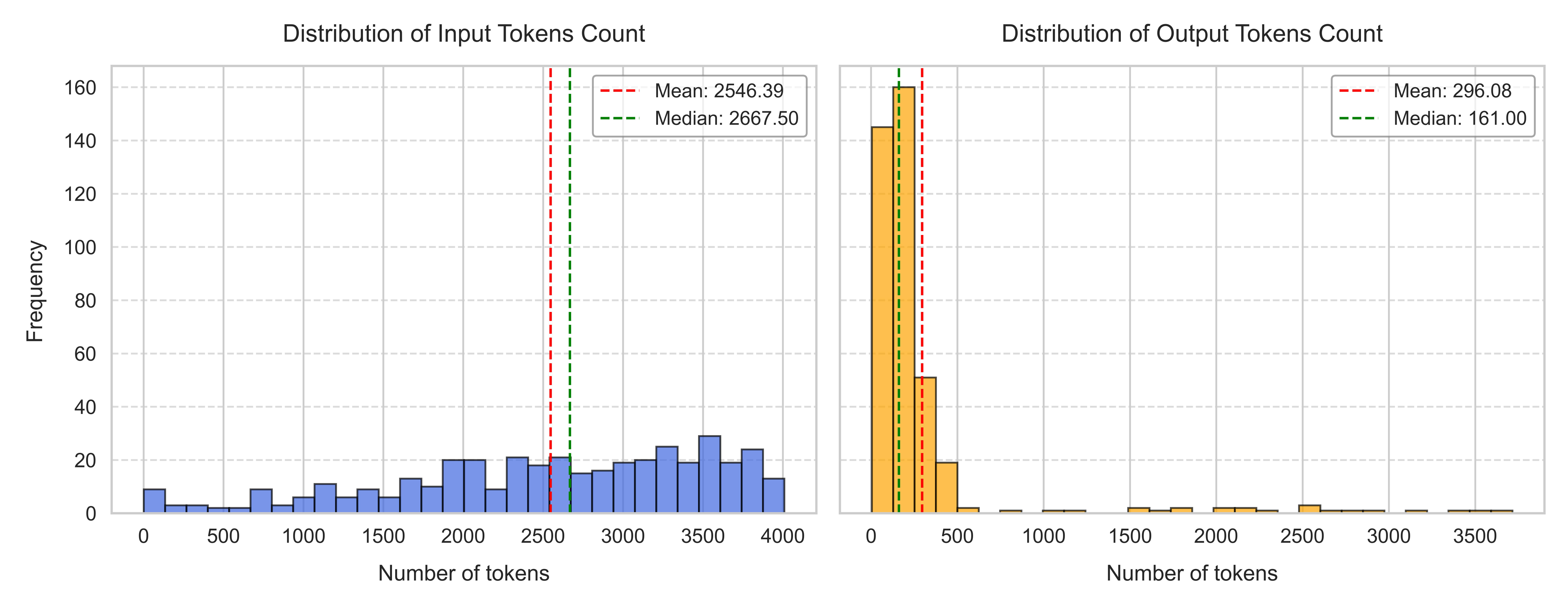}
\caption{\textbf{Upper: } Distribution of the number of tokens of input prompt and output response respectively in the data \cite{zheng2023lmsys}. \textbf{Lower: } Distribution of the number of tokens of input prompt and output response respectively in the data \cite{cohan2018discourse}.} 
\label{fig:distribution1}
\end{figure}

The second dataset, constructed by \cite{cohan2018discourse}, focuses on long-form summarization of academic papers from Arxiv (available at \url{https://github.com/armancohan/long-summarization}). Here, prompts are significantly longer, with a mean input length of 2546.39 tokens (median: 2667.50) and a mean output length of 296.08 tokens (median: 161.00), illustrated in Figure \ref{fig:distribution1}.

From these datasets, we randomly select 1600 short conversations and 400 long summarization tasks (we also test other mixing ratio ablations in this section). By merging them with a random permutation, we create a mixed dataset of 2000 samples. Figure \ref{fig:distribution3} reveals the resulting distribution: the input lengths now exhibit a mean of 542.75 tokens (median: 18.00), while outputs have a mean of 127.50 tokens, demonstrating the desired variability in prompt sizes.  

\xhdr{Experiment Setup. }
We evaluate performance using average end-to-end latency per request, defined as the total latency across all requests divided by the number of requests. To address randomness in workload sampling and request ordering, all main experiments are repeated over five random seeds. We report the mean and standard deviation across these seeds.

For benchmarking, we compare against the following scheduling policies:
\begin{enumerate}
    \item \texttt{FCFS}: Requests are processed in the random arrival order. Each batch is filled maximally without exceeding the KV-cache memory limit, assuming perfect knowledge of output lengths.
    \item \mcsdf{}: As described in Section~\ref{subsec:sf}, this policy prioritizes requests with shorter output lengths.
    \item \sortLP{}: Introduced in Section~\ref{sec:extension}, \sortLP{} solves a relaxed linear program to estimate request starting times and batches requests in ascending order of these values.
    \item \LPSwap{}: The hybrid method outlined in Section~\ref{sec:extension}, combining \sortLP{} ordering with F-metric-guided local swaps.
    \item \sortF{} variants: Since exact dynamic programming is computationally infeasible for the main experiment with \(n=2000\), we implement \sortF{} using the Local Swap Search and Quantile Greedy Selection Phase~1 solvers described in Section~\ref{sec:approximation}.
\end{enumerate}

For scalability analysis, we evaluate latency trends by subsampling the mixed dataset at \(n=200,400,\ldots,2000\). The baseline workload composition follows the main mixed dataset described above, with \(80\%\) short conversational requests and \(20\%\) long summarization requests. To evaluate when heterogeneous scheduling matters most, we further vary the short/long workload composition across \(100/0\), \(80/20\), \(50/50\), and \(20/80\). We keep the KV-cache memory budget fixed at \(M=16{,}492\) tokens, corresponding to the LLaMA2-70B deployment on two A100 GPUs considered in this section. This reflects the fact that \(M\) is determined by the hardware and model configuration, while the effective memory pressure varies with workload composition.

Batch processing times are estimated using the linear-regression timing model in the Vidur simulator from \cite{agrawal2024vidur}. In addition to average latency, we report p95 and p99 latency to evaluate tail performance. The LP relaxation supplies a benchmark for the unit-time model. Its role and its distinction from the Vidur wall-clock evaluation are specified below.

\xhdr{Main Results. }
Figure~\ref{fig:main_latency_variance} reports average end-to-end latency per request under the \(80/20\) short/long mixture, with error bars representing one standard deviation over five seeds. The ranking of algorithms is stable across random seeds. At \(n=2000\), \sortF{} with Local Swap Search achieves the lowest average latency, \(154.1 \pm 2.3\) seconds, closely followed by \LPSwap{} with \(157.4 \pm 1.6\) seconds. Both substantially outperform \texttt{FCFS} (\(750.0 \pm 29.8\) seconds), \mcsdf{} (\(321.5 \pm 2.0\) seconds), \sortLP{} (\(235.1 \pm 1.6\) seconds), and \sortF{} with Quantile Greedy Selection (\(243.7 \pm 2.7\) seconds).

\begin{figure}[!h]
\center
\includegraphics[width=0.7\textwidth]{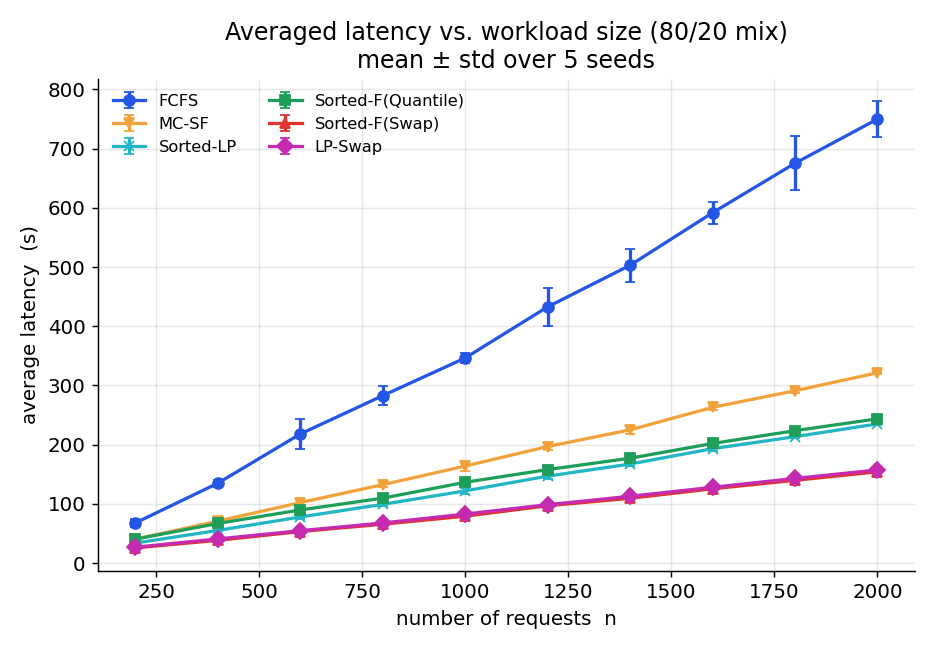}
\caption{Average end-to-end latency per request under the \(80/20\) short/long workload mixture. Error bars report mean \(\pm\) one standard deviation over five random seeds.}
\label{fig:main_latency_variance}
\end{figure}

These results reinforce the theoretical message of the paper. The inferior performance of \mcsdf{} shows that purely output-length-based prioritization is inadequate when prompt sizes vary substantially. In contrast, the F-metric explicitly balances batch cardinality against downstream decode cost, leading to consistently lower average latency. The close overlap between \sortF{} with Local Swap Search and \LPSwap{} indicates that the F-metric-guided refinement dominates the final schedule quality: after local swaps, the LP-based initialization provides little additional benefit. Since \sortF{} avoids solving the LP and has a simpler implementation, it remains the preferred practical method.

\xhdr{Sensitivity to Workload Composition. }
Figure~\ref{fig:composition_sensitivity} shows how the benefit of \sortF{} changes with the short/long workload mixture. The advantage of \sortF{} is largest under the mixed \(80/20\) workload, where \sortF{} with Local Swap Search improves over \texttt{FCFS} by \(4.87\times\) and over \mcsdf{} by \(2.09\times\) at \(n=2000\). Among the tested compositions, the largest speedup occurs under
the $80/20$ mixture. The improvement over \texttt{FCFS} is
$3.36\times$ under the $100/0$ workload, $2.50\times$ under the
$50/50$ workload, and $1.65\times$ under the $20/80$ workload.
These results indicate that the scheduling benefit depends on
workload composition and does not vary monotonically with the
fraction of long requests.

\begin{figure}[!h]
\center
\includegraphics[width=0.95\textwidth]{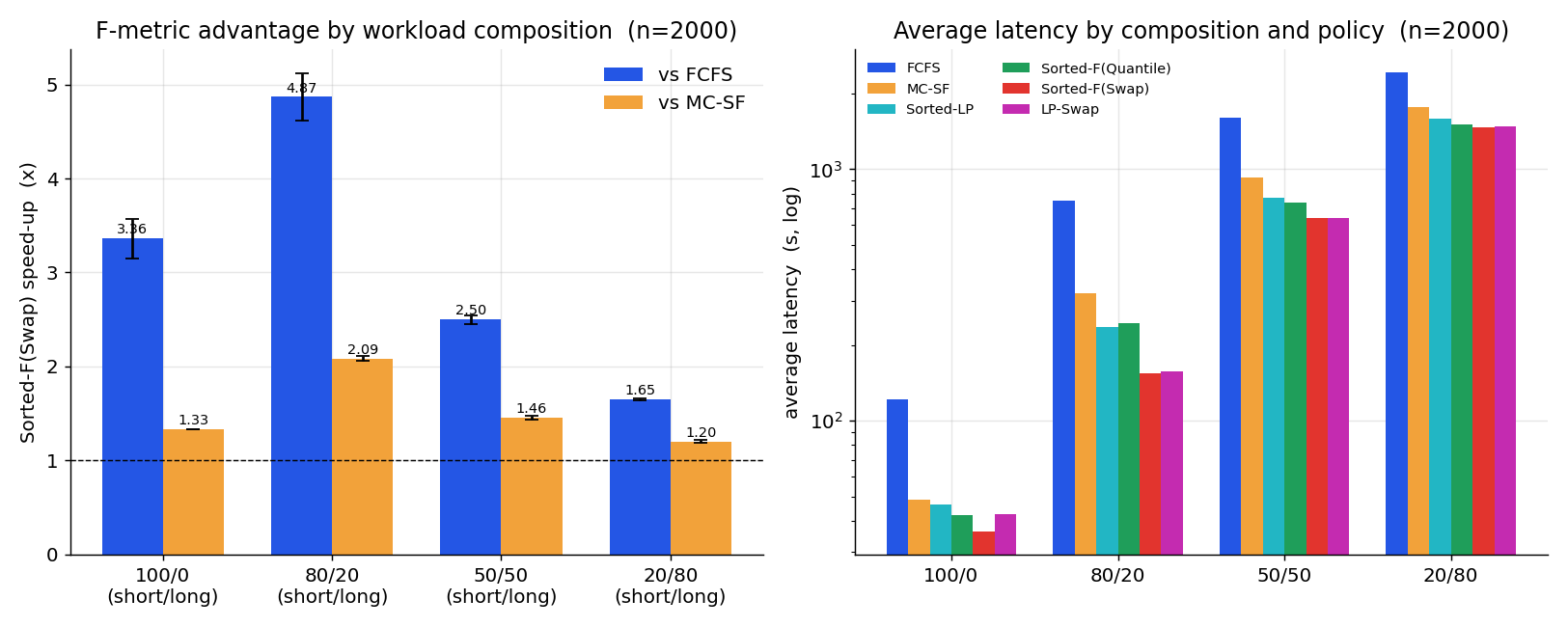}
\caption{Sensitivity of scheduling performance to workload composition at \(n=2000\). Left: speedup of \sortF{} with Local Swap Search over \texttt{FCFS} and \mcsdf{}. Right: average latency by policy and workload composition.}
\label{fig:composition_sensitivity}
\end{figure}

This pattern identifies the regime in which heterogeneous scheduling matters most. When all requests are short, most policies can pack requests efficiently. When long requests dominate, memory becomes the binding bottleneck for nearly all policies, leaving less room for scheduling improvements. The F-metric is most valuable in mixed workloads, where the scheduler must trade off many short requests against a smaller number of memory-intensive long-context requests. This observation provides a practical diagnostic: operators should expect the largest gains from memory-aware scheduling when request lengths have high dispersion and the server operates close to its KV-cache capacity.

\xhdr{LP Relaxation Lower Bound.}
To evaluate the empirical solution quality of \sortF{} with Local
Swap Search, we solve the LP relaxation on smaller instances where
the relaxation is computationally tractable. Since the LP relaxation
provides a lower bound on the optimal offline total latency, the
ratio between the TEL achieved by \sortF{} with Local Swap Search
and the LP objective gives an instance-specific upper bound on its
ratio to the optimum under the same unit-time decoding model.

\begin{figure}[!h]
\centering
\includegraphics[width=0.7\textwidth]{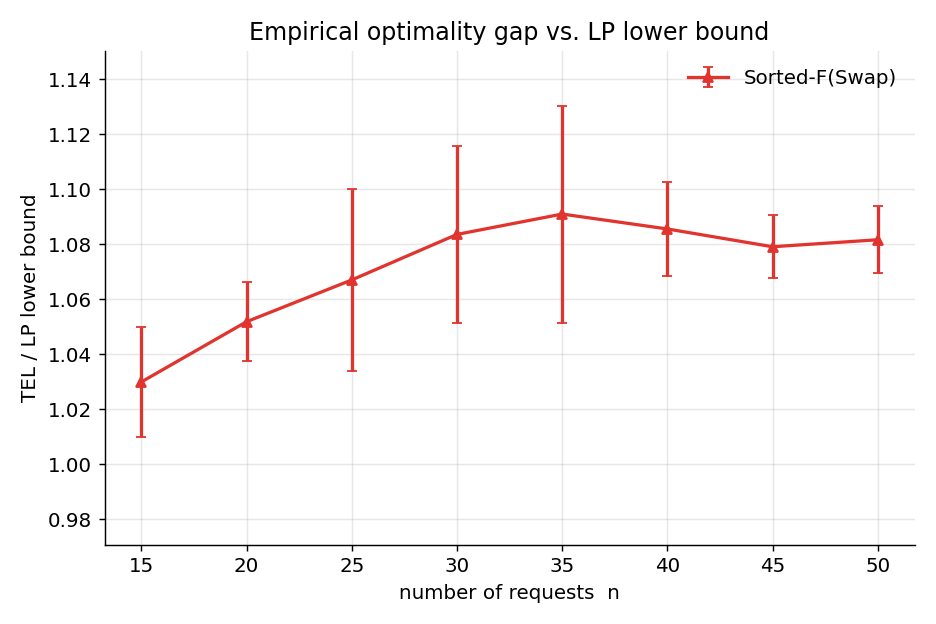}
\caption{Empirical performance of \sortF{} with Local Swap Search
relative to the LP relaxation lower bound on small tractable
instances.}
\label{fig:lp_optimality_gap}
\end{figure}

As shown in Figure~\ref{fig:lp_optimality_gap}, the empirical ratio
remains close to one across the tested instance sizes. In particular,
\sortF{} with Local Swap Search achieves average ratios of
approximately $1.03$--$1.09$ relative to the LP lower bound.
These results indicate that the implemented heuristic performs
close to the optimum on the tested instances in the unit-time
decoding model. This empirical finding complements the factor-$40$
guarantee in Theorem~\ref{thm:cr}, which applies to exact Phase~1
selection; it does not establish the same worst-case guarantee
for Local Swap Search.

\xhdr{Tail Latency. }
Figure~\ref{fig:tail_latency} in Appendix \ref{append:num} reports p99 latency as the workload size grows and compares average, p95, and p99 latency across policies at \(n=2000\). \sortF{} with Local Swap Search substantially reduces average latency and also improves p95 latency relative to \texttt{FCFS}. At \(n=2000\), its p95 latency is \(953.2\) seconds, compared with \(1413.9\) seconds for \texttt{FCFS}.

\xhdr{Scheduling Runtime. }
Figure~\ref{fig:runtime} reports the ordering-construction time of each policy. All \sortF{} variants remain computationally lightweight at the main experimental scale. Under the \(80/20\) mixture with \(n=2000\), Quantile Greedy Selection takes approximately \(0.053\) seconds and Local Swap Search takes approximately \(0.073\) seconds to construct the scheduling order. This runtime is small relative to the total serving horizon, suggesting that F-metric-based scheduling can be implemented as a low-overhead decision layer in backlogged serving regimes.

\begin{figure}[!h]
\center
\includegraphics[width=0.65\textwidth]{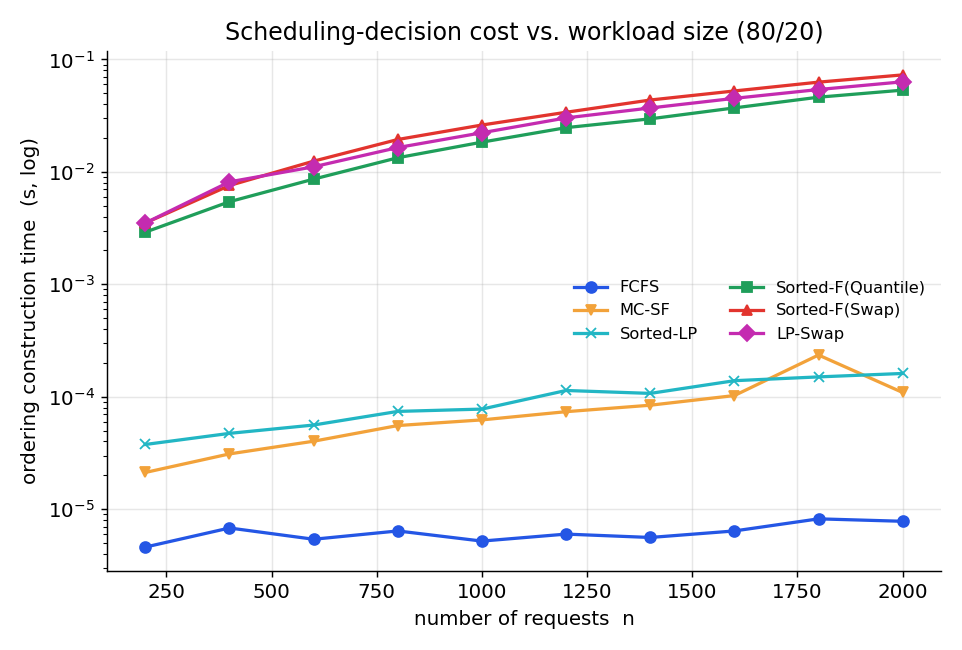}
\caption{Scheduling-order construction time under the \(80/20\) workload mixture. The y-axis is shown on a logarithmic scale.}
\label{fig:runtime}
\end{figure}

\xhdr{Robustness to prefill weighting.} To test whether explicitly charging prefill cost in the selection metric improves scheduling, we evaluate the generalized metric $G(\mathcal{X})=\sum_{r_i\in\mathcal X}(\alpha s_i+o_i)/|\mathcal{X}|^2$, where $\alpha$ is the weight of a prefill token relative to a decode token and $\alpha=0$ recovers the $F$-metric (Figure \ref{fig:sortg}). When $\alpha$ is calibrated to the true relative compute cost of prefill in our timing model---the \emph{compute-matched} value $\alpha\approx 0.013$, reflecting that prefill ingests tokens in parallel while decode is sequential---Sorted-G and Sorted-F have similar reported mean latency over five seeds: at the prompt lengths in our workload (up to a few thousand tokens), prefill compute is a second-order effect and the reported results show little benefit from adding the calibrated prefill term. Latency degrades as $\alpha$ rises above its compute-matched value: by about $5.5\%$ at the \emph{memory-matched} value $\alpha=1$, where $G(\mathcal X)=\sum_{i\in\mathcal X}(s_i+o_i)/|\mathcal X|^2$. This still depends on batch cardinality and is not the individual priority rule analyzed in the \mcsdff{} counterexample. Latency degrades by up to $\sim 19\%$ for $\alpha\ge 5$, where prefill length receives substantially greater weight
relative to decode length in the batch-selection objective; the observed saturation is consistent with the prefill term increasingly determining the order. These results support the decode-centric modeling choice in the regime we study. 

Theorem~\ref{thm:cr} concerns $F$ in the unit-time decoding model. Its proof uses the bound $F(\mathcal B)\le D(\mathcal B)/|\mathcal B|$; the additional prefill term in $G$ is not uniformly controlled by the decode duration $D(\mathcal B)$. More specifically, this is the step used in Lemma~4 to convert exact Phase~1 minimization into the deadline inequality $q_k(t)f_k\le 5t$. For the generalized metric $G$, the additional term $\alpha\sum_{i\in B}s_i/|B|^2$ has no analogous bound in terms of the decode-duration quantity $D(B)$ under the present model, so the residual-set lower-bound argument does not extend directly. Accordingly, the theorem does not supply a factor-$40$ guarantee for arbitrary prefill weights or Vidur wall-clock times. This experiment evaluates those extensions empirically.

\subsection{Experiment under Online Poisson Arrivals} \label{sec:num_online}
Finally, we conduct a preliminary online experiment to evaluate whether the F-metric remains useful when requests arrive over time. This experiment is intended as an empirical stress test of the scheduling principle rather than an online competitive-ratio analysis. Online latency is measured from each request's arrival to its completion, $c_i-a_i$. Requests arrive according to a Poisson process over a horizon of \(T=4000\) iterations. Each arriving request is sampled from the same \(80/20\) short/long mixture used in the main offline experiment. We consider arrival rates \(\lambda\in\{0.05,0.07,0.09,0.11,0.13,0.15,0.18,0.21,0.25,0.30\}\), corresponding to normalized loads \(\rho=\lambda/\lambda^*\) ranging from \(0.48\) to \(2.88\), where \(\lambda^*=0.104\) is the memory-bound throughput estimate. Arrivals stop at time \(T\), and the system is then drained to completion, so all policies are evaluated on the same realized request set for each seed.

We compare three online policies. \texttt{FCFS} admits requests in arrival order. \mcsdf{} recomputes the shortest-decode-first priority over the current backlog at every iteration. \sortF{}-Online is the receding-horizon policy described in Section~\ref{sec:extension}: it recomputes the F-based priority order every \(10\) iterations, or earlier when the cached order is exhausted or when newly arrived pending requests account for at least \(50\%\) of the backlog size at the previous recomputation. The F-based planning step is capped at a look-ahead window of \(300\) pending requests, and feasibility is checked against the live residual KV-cache profile at every iteration.

\begin{figure}[!h]
\center
\includegraphics[width=0.65\textwidth]{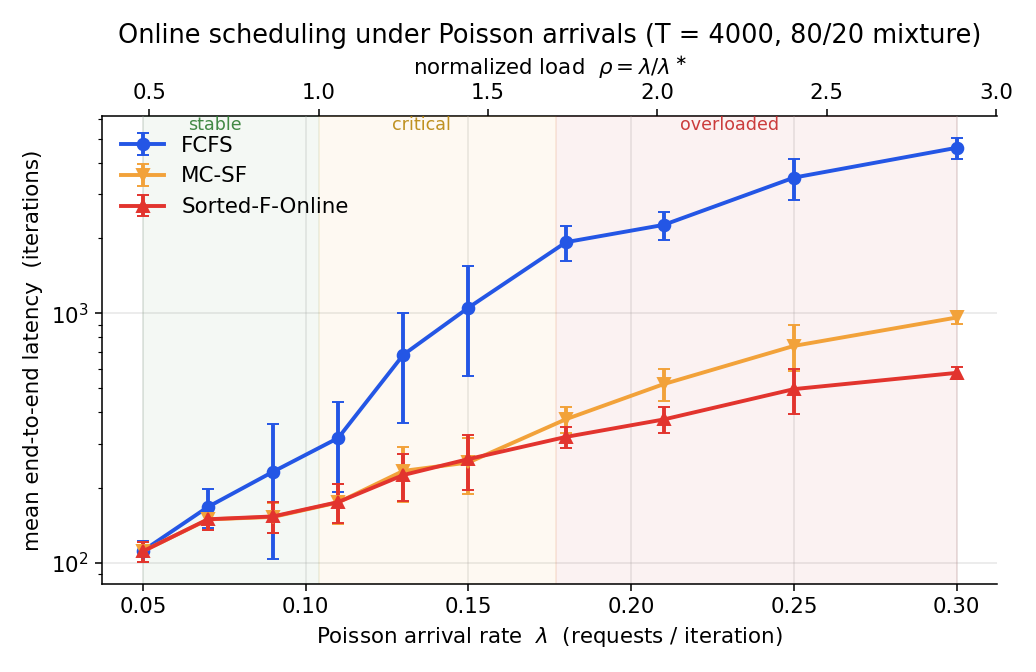}
\caption{Online scheduling under Poisson arrivals. The y-axis reports mean end-to-end latency in iterations, with error bars over five random seeds. The top axis shows normalized load \(\rho=\lambda/\lambda^*\).}
\label{fig:online_latency}
\end{figure}

Figure~\ref{fig:online_latency} shows that the benefit of \sortF{}-Online is regime-dependent. In the stable and lightly critical regimes, where the backlog is shallow, \sortF{}-Online and \mcsdf{} perform similarly because there is limited scheduling flexibility. At higher offered loads, a deeper backlog provides more scope for
batch selection to affect performance. For the tested loads
$\rho\ge 1.73$, \sortF{}-Online achieves lower mean latency than
the online \mcsdf{} baseline. Its mean-latency reduction relative
to \mcsdf{} is $15.1\%$ at $\rho=1.73$ and reaches $40.0\%$ at
$\rho=2.88$. Relative to \texttt{FCFS}, the reduction is
approximately $83\%$--$88\%$ over these higher-load settings.
Near the nominal load threshold, the advantage over \mcsdf{}
is smaller and is not uniform across the tested arrival rates.

Figure~\ref{fig:online_gap} in Appendix \ref{append:num} further illustrates that the advantage of F-metric-based scheduling grows with overload. This finding is consistent with the motivation of our offline model: during peak-demand or transient-overload periods, the system may face a substantial backlog, and choosing the right feasible batches becomes more important than simply prioritizing short outputs. At the same time, the experiment also highlights the limitation of the current paper. In stable regimes, where requests can often be served shortly after arrival, sophisticated backlog scheduling provides little additional benefit over a strong online \mcsdf{} baseline. A formal online model, including competitive-ratio analysis or wall-clock evaluation under production traces, remains an important direction for future work.

\section{Conclusion} \label{sec:conclusion}

This paper studies offline scheduling for LLM inference on a single worker under a fixed KV-cache memory budget. Motivated by mixed workloads with heterogeneous prefill and decode lengths, we show that relaxing the uniform-input-size assumption fundamentally alters the problem: minimizing total end-to-end latency becomes NP-hard, and the shortest-output and shortest-total-size priority rules can incur unbounded approximation ratios. Variable prefill lengths thus cannot be treated as a minor modeling detail in memory-constrained LLM serving.

To address this, we propose \sortF{}, an offline algorithm built on a batch-level quality metric $F(\mathcal{X})$ that balances batch cardinality against downstream decode cost. When the Phase~1 batch-selection subproblem is solved exactly, \sortF{} achieves a factor-$40$ approximation in the unit-time, uninterrupted-decoding model with known output lengths, under either the dynamic synchronous-feasibility rule or the static peak-memory surrogate. The guarantee holds for every finite instance.

On mixed real-world workloads, the F-metric consistently lowers latency relative to standard baselines, and a preliminary Poisson-arrival experiment indicates that a receding-horizon variant remains effective in overloaded online regimes.

\subsection{Managerial Insights}

Our results suggest several operational implications for LLM serving systems with heterogeneous prefill and decode lengths.

First, heterogeneous scheduling matters most when request lengths are highly dispersed and memory is tight. If all requests have similar prefill sizes $s_i$ and decode lengths $o_i$, then simple rules such as shortest-output-first may already capture much of the scheduling structure. In contrast, when the coefficients of variation of $s_i$ and/or $o_i$ are large, a scheduler must trade off short jobs, large-prefill jobs, and long-decode jobs competing for the same KV-cache budget. This is precisely the regime in which myopic rules can become misleading: a request with a short decode length may occupy a large amount of memory, while a request with a long decode length may be cheap to keep active. Practically, this suggests that operators should pay particular attention to workload mixes combining short conversational queries, long-context document tasks, and summarization requests, especially when the worker operates close to its memory capacity.

Second, the choice of the Phase~1 solver should depend on the scale and time sensitivity of the deployment. The exact dynamic programming implementation solves the static peak-memory subproblem exactly and inherits the factor-$40$ guarantee described above, but it is appropriate only for small or moderate residual sets because its complexity is pseudopolynomial in $M$. Local Swap Search provides a useful middle ground: it refines the initial greedy selection and has modest scheduling overhead in the reported experiments. Quantile Greedy Selection offers $O(r\log r)$ time per batch-selection call and avoids repeated swap scans, making it useful for large backlogs. Thus, a practical deployment can use exact dynamic programming for small backlogs, Local Swap Search when local refinement is computationally affordable, and Quantile Greedy Selection for large peak-hour backlogs. This also suggests a hybrid implementation: the system can switch among the three solvers according to the current queue length, the available scheduling time budget, and the service-level objective. Such a hybrid is a heuristic unless exact selection is used at every Phase~1 iteration.

Third, output-length prediction is valuable, but its value should be interpreted through the scheduler's robustness to prediction error. A highly accurate prediction model can move \sortF{} closer to the oracle setting and improve the quality of the Phase~1 ordering. Nevertheless, prediction accuracy should be evaluated jointly with the scheduling policy, rather than in isolation. If the prediction model is costly to train or maintain, the relevant question is whether its improved ordering leads to meaningful latency reduction after accounting for prediction overhead. From an operational perspective, this suggests that practitioners should assess the value of prediction models by measuring their downstream impact on latency, rather than only by prediction error metrics.

\subsection{Limitations and Future Work}

A key limitation of the present work is that our theoretical model and approximation guarantee are developed in an offline setting, where all requests are assumed to be available at time $t=0$. This setting is motivated by peak-hour or tidal-demand scenarios in which a server faces an accumulated backlog of pending requests, but it does not fully capture online serving environments where requests arrive sequentially and decisions must be made without knowledge of future arrivals. In Section~\ref{sec:num}, we provide a preliminary Poisson-arrival experiment for a receding-horizon version of \sortF{}, showing that the $F$-metric remains useful in overloaded online regimes. Nevertheless, a formal online theory remains beyond the scope of this paper. Developing online performance guarantees is an important direction for future work.

Several additional directions remain open. First, the factor-$40$ guarantee applies to the idealized version of \sortF{} in which Phase~1 is solved exactly over either the dynamic synchronous-feasibility family or the static peak-memory family. The static dynamic program satisfies this requirement, whereas Local Swap Search and Quantile Greedy Selection do not have established worst-case approximation guarantees. Second, our LP relaxation provides a useful lower bound and ordering signal, but the integrality gap of the relaxation and the performance of LP-based rounding procedures are not yet fully understood. Finally, the theoretical analysis assumes known decode lengths and uninterrupted unit-time decoding. Developing scheduling guarantees under prediction uncertainty and more detailed service-cost models would further improve the applicability of the model.

\bibliographystyle{ACM-Reference-Format}
\bibliography{reference}

@article{jaillet2025online,
  title = {Online Scheduling for {LLM} Inference with {KV} Cache Constraints},
  author = {Jaillet, Patrick and Jiang, Jiashuo and Mellou, Konstantina and Molinaro, Marco and Podimata, Chara and Zhou, Zijie},
  journal = {arXiv preprint arXiv:2502.07115},
  year = {2025}
}

@article{ao2025optimizing,
  title = {Optimizing {LLM} Inference: Fluid-Guided Online Scheduling with Memory Constraints},
  author = {Ao, Ruicheng and Luo, Gan and Simchi-Levi, David and Wang, Xinshang},
  journal = {arXiv preprint arXiv:2504.11320},
  year = {2025}
}

@article{shahout2024don,
  title = {Don't Stop Me Now: Embedding Based Scheduling for {LLMs}},
  author = {Shahout, Rana and Malach, Eran and Liu, Chunwei and Jiang, Weifan and Yu, Minlan and Mitzenmacher, Michael},
  journal = {arXiv preprint arXiv:2410.01035},
  year = {2024}
}

@article{allahverdi2008survey,
  title = {A survey of scheduling problems with setup times or costs},
  author = {Allahverdi, Ali and Ng, Chi To and Cheng, TC Edwin and Kovalyov, Mikhail Y},
  journal = {European journal of operational research},
  volume = {187},
  number = {3},
  pages = {985--1032},
  year = {2008},
  publisher = {Elsevier}
}

@article{brucker1998scheduling,
  title = {Scheduling a batching machine},
  author = {Brucker, Peter and Gladky, Andrei and Hoogeveen, Han and Kovalyov, Mikhail Y and Potts, Chris N and Tautenhahn, Thomas and Van De Velde, Steef L},
  journal = {Journal of scheduling},
  volume = {1},
  number = {1},
  pages = {31--54},
  year = {1998},
  publisher = {Wiley Online Library}
}

@incollection{chen1998review,
  title = {A review of machine scheduling: Complexity, algorithms and approximability},
  author = {Chen, Bo and Potts, Chris N and Woeginger, Gerhard J},
  pages = {1493--1641},
  year = {1998},
  publisher = {Springer},
  booktitle = {Handbook of Combinatorial Optimization},
  editor = {Du, Ding-Zhu and Pardalos, Panos M.},
  doi = {10.1007/978-1-4613-0303-9_25}
}

@article{zheng2023lmsys,
  title = {{LMSYS-Chat-1M}: A large-scale real-world {LLM} conversation dataset},
  author = {Zheng, Lianmin and Chiang, Wei-Lin and Sheng, Ying and Li, Tianle and Zhuang, Siyuan and Wu, Zhanghao and Zhuang, Yonghao and Li, Zhuohan and Lin, Zi and Xing, Eric and others},
  journal = {arXiv preprint arXiv:2309.11998},
  year = {2023}
}

@article{agrawal2024vidur,
  title = {Vidur: A Large-Scale Simulation Framework For {LLM} Inference},
  author = {Agrawal, Amey and Kedia, Nitin and Mohan, Jayashree and Panwar, Ashish and Kwatra, Nipun and Gulavani, Bhargav and Ramjee, Ramachandran and Tumanov, Alexey},
  journal = {Proceedings of Machine Learning and Systems},
  volume = {6},
  pages = {351--366},
  year = {2024}
}

@article{agrawal2024taming,
  title = {Taming throughput-latency tradeoff in {LLM} inference with {Sarathi-Serve}},
  author = {Agrawal, Amey and Kedia, Nitin and Panwar, Ashish and Mohan, Jayashree and Kwatra, Nipun and Gulavani, Bhargav S and Tumanov, Alexey and Ramjee, Ramachandran},
  journal = {arXiv preprint arXiv:2403.02310},
  year = {2024}
}

@article{cohan2018discourse,
  title = {A discourse-aware attention model for abstractive summarization of long documents},
  author = {Cohan, Arman and Dernoncourt, Franck and Kim, Doo Soon and Bui, Trung and Kim, Seokhwan and Chang, Walter and Goharian, Nazli},
  journal = {arXiv preprint arXiv:1804.05685},
  year = {2018}
}

@article{agrawal2023sarathi,
  title = {Sarathi: Efficient {LLM} inference by piggybacking decodes with chunked prefills},
  author = {Agrawal, Amey and Panwar, Ashish and Mohan, Jayashree and Kwatra, Nipun and Gulavani, Bhargav S and Ramjee, Ramachandran},
  journal = {arXiv preprint arXiv:2308.16369},
  year = {2023}
}

@inproceedings{patel2023splitwise,
  title = {{Splitwise}: Efficient Generative {LLM} Inference Using Phase Splitting},
  author = {Patel, Pratyush and Choukse, Esha and Zhang, Chaojie and Shah, Aashaka and Goiri, {\'I}{\~n}igo and Maleki, Saeed and Bianchini, Ricardo},
  pages = {118--132},
  year = {2024},
  booktitle = {Proceedings of the 51st Annual International Symposium on Computer Architecture (ISCA)},
  doi = {10.1109/ISCA59077.2024.00019}
}

@article{qiu2024efficient,
  title = {Efficient interactive {LLM} serving with proxy model-based sequence length prediction},
  author = {Qiu, Haoran and Mao, Weichao and Patke, Archit and Cui, Shengkun and Jha, Saurabh and Wang, Chen and Franke, Hubertus and Kalbarczyk, Zbigniew T and Ba{\c{s}}ar, Tamer and Iyer, Ravishankar K},
  journal = {arXiv preprint arXiv:2404.08509},
  year = {2024}
}

@article{brown2020language,
  title = {Language models are few-shot learners},
  author = {Brown, Tom and Mann, Benjamin and Ryder, Nick and Subbiah, Melanie and Kaplan, Jared D and Dhariwal, Prafulla and Neelakantan, Arvind and Shyam, Pranav and Sastry, Girish and Askell, Amanda and others},
  journal = {Advances in neural information processing systems},
  volume = {33},
  pages = {1877--1901},
  year = {2020}
}

@article{openai2023gpt,
  title = {{GPT}-4 Technical Report},
  author = {{OpenAI}},
  journal = {arXiv preprint arXiv:2303.08774},
  year = {2023},
  doi = {10.48550/arXiv.2303.08774}
}

@misc{anthropic2023claude,
  author = {{Anthropic}},
  title = {Claude},
  howpublished = {\url{https://claude.ai}},
  year = {2023}
}

@article{cascella2023evaluating,
  title = {Evaluating the feasibility of {ChatGPT} in healthcare: an analysis of multiple clinical and research scenarios},
  author = {Cascella, Marco and Montomoli, Jonathan and Bellini, Valentina and Bignami, Elena},
  journal = {Journal of medical systems},
  volume = {47},
  number = {1},
  pages = {33},
  year = {2023},
  publisher = {Springer}
}

@inproceedings{kwon2023efficient,
  title = {Efficient memory management for large language model serving with {PagedAttention}},
  author = {Kwon, Woosuk and Li, Zhuohan and Zhuang, Siyuan and Sheng, Ying and Zheng, Lianmin and Yu, Cody Hao and Gonzalez, Joseph and Zhang, Hao and Stoica, Ion},
  booktitle = {Proceedings of the 29th Symposium on Operating Systems Principles},
  pages = {611--626},
  year = {2023}
}

@inproceedings{yu2022orca,
  title = {{Orca}: A Distributed Serving System for {Transformer-Based} Generative Models},
  author = {Yu, Gyeong-In and Jeong, Joo Seong and Kim, Geon-Woo and Kim, Soojeong and Chun, Byung-Gon},
  booktitle = {16th USENIX Symposium on Operating Systems Design and Implementation (OSDI 22)},
  pages = {521--538},
  year = {2022}
}

@incollection{albers2009online,
  title = {Online scheduling},
  author = {Albers, Susanne},
  booktitle = {Introduction to scheduling},
  pages = {51--77},
  year = {2009},
  publisher = {CRC Press}
}

@article{schedPrecedence,
  title = {On-line scheduling with precedence constraints},
  journal = {Discrete Applied Mathematics},
  volume = {119},
  number = {1},
  pages = {169--180},
  year = {2002},
  issn = {0166-218X},
  doi = {10.1016/S0166-218X(01)00272-4},
  url = {https://www.sciencedirect.com/science/article/pii/S0166218X01002724},
  author = {Yossi Azar and Leah Epstein}
}

@article{li2020online,
  title = {Online batch scheduling of simple linear deteriorating jobs with incompatible families},
  author = {Li, Wenhua and Wang, Libo and Chai, Xing and Yuan, Hang},
  journal = {Mathematics},
  volume = {8},
  number = {2},
  pages = {170},
  year = {2020},
  publisher = {MDPI}
}

@article{vaswani2017attention,
  title = {Attention is all you need},
  author = {Vaswani, Ashish and Shazeer, Noam and Parmar, Niki and Uszkoreit, Jakob and Jones, Llion and Gomez, Aidan N and Kaiser, {\L}ukasz and Polosukhin, Illia},
  journal = {Advances in neural information processing systems},
  volume = {30},
  year = {2017}
}

@article{chen2008logistics,
  title = {Logistics scheduling with batching and transportation},
  author = {Chen, Bo and Lee, Chung-Yee},
  journal = {European journal of operational research},
  volume = {189},
  number = {3},
  pages = {871--876},
  year = {2008},
  publisher = {Elsevier}
}

@article{brucker1999resource,
  title = {Resource-constrained project scheduling: Notation, classification, models, and methods},
  author = {Brucker, Peter and Drexl, Andreas and M{\"o}hring, Rolf and Neumann, Klaus and Pesch, Erwin},
  journal = {European journal of operational research},
  volume = {112},
  number = {1},
  pages = {3--41},
  year = {1999},
  publisher = {Elsevier}
}

@article{zheng2023response,
  title = {Response Length Perception and Sequence Scheduling: An {LLM}-Empowered {LLM} Inference Pipeline},
  author = {Zheng, Zangwei and Ren, Xiaozhe and Xue, Fuzhao and Luo, Yang and Jiang, Xin and You, Yang},
  journal = {Advances in Neural Information Processing Systems},
  volume = {36},
  pages = {65517--65530},
  year = {2023}
}

@article{chen2025adaptively,
  title = {Adaptively robust {LLM} inference optimization under prediction uncertainty},
  author = {Chen, Zixi and Ye, Yinyu and Zhou, Zijie},
  journal = {arXiv preprint arXiv:2508.14544},
  year = {2025}
}

@article{nie2026queueing,
  title = {A Queueing-Theoretic Framework for Stability Analysis of {LLM} Inference with {KV} Cache Memory Constraints},
  author = {Nie, Chengyi and Si, Nian and Zhou, Zijie},
  journal = {arXiv preprint arXiv:2605.04595},
  year = {2026}
}

@article{zhou2026position,
  title = {Position: {LLM} Serving Needs Mathematical Optimization and Algorithmic Foundations, Not Just Heuristics},
  author = {Zhou, Zijie},
  journal = {arXiv preprint arXiv:2605.01280},
  year = {2026}
}

@article{bu2026tackling,
  title = {Tackling the Data-Parallel Load Balancing Bottleneck in {LLM} Serving: Practical Online Routing at Scale},
  author = {Bu, Tianci and Lyu, Yuan and Chen, Zixi and Song, Chendong and Liang, Hong and Gurung, Tsepten and Fan, Yuwei and Ye, Yinyu and Zhou, Zijie},
  journal = {arXiv preprint arXiv:2605.06113},
  year = {2026}
}

@article{chen2026universal,
  title = {A Universal Load Balancing Principle and Its Application to Large Language Model Serving},
  author = {Chen, Zixi and Bu, Tianci and Song, Chendong and Lu, Xin and Ye, Yinyu and Zhou, Zijie},
  journal = {arXiv preprint arXiv:2601.17855},
  year = {2026}
}

@book{gareyjohnson1979,
  author = {Garey, Michael R. and Johnson, David S.},
  title = {Computers and Intractability: A Guide to the Theory of {NP}-Completeness},
  series = {A Series of Books in the Mathematical Sciences},
  publisher = {W. H. Freeman and Company},
  address = {San Francisco},
  year = {1979},
  isbn = {0-7167-1044-7}
}

@misc{microsoftcopilot2026,
  author = {{Microsoft}},
  title = {{Microsoft Copilot}},
  howpublished = {\url{https://copilot.microsoft.com/}},
  year = {2026},
  note = {Accessed June 16, 2026}
}

@misc{githubcopilot2026,
  author = {{GitHub}},
  title = {{GitHub Copilot}},
  howpublished = {\url{https://github.com/features/copilot}},
  year = {2026},
  note = {Accessed June 16, 2026}
}

@article{wang2024burstgpt,
  title = {Towards Efficient and Reliable {LLM} Serving: A Real-World Workload Study},
  author = {Wang, Yuxin and Chen, Yuhan and Li, Zeyu and Tang, Zhenheng and Guo, Rui and Wang, Xin and Wang, Qiang and Zhou, Amelie Chi and Chu, Xiaowen},
  journal = {arXiv preprint arXiv:2401.17644v1},
  year = {2024},
  url = {https://arxiv.org/abs/2401.17644v1}
}

@inproceedings{goel2026qoserve,
  title = {QoServe: Breaking the Silos of {LLM} Inference Serving},
  author = {Goel, Kanishk and Mohan, Jayashree and Kwatra, Nipun and Anupindi, Ravi Shreyas and Ramjee, Ramachandran},
  booktitle = {Proceedings of the 31st ACM International Conference on Architectural Support for Programming Languages and Operating Systems, Volume 2 (ASPLOS)},
  year = {2026},
  doi = {10.1145/3779212.3790206}
}

@article{jiang2026oserve,
  title = {OServe: Accelerating {LLM} Serving via Spatial-Temporal Workload Orchestration},
  author = {Jiang, Youhe and Fu, Fangcheng and Wang, Taiyi and He, Guoliang and Yoneki, Eiko},
  journal = {arXiv preprint arXiv:2602.12151},
  year = {2026}
}

@article{feige2004approximating,
  author = {Feige, Uriel and Lov{\'a}sz, L{\'a}szl{\'o} and Tetali, Prasad},
  title = {Approximating Min Sum Set Cover},
  journal = {Algorithmica},
  year = {2004},
  volume = {40},
  number = {4},
  pages = {219--234},
  doi = {10.1007/s00453-004-1110-5}
}

\ECSwitch

\ECHead{Appendix: Proofs of Statements}

\section{Other Related Work} \label{sec:other}

\xhdr{LLM Inference. } Improving the efficiency of LLM inference is crucial due to the substantial computational and financial costs it incurs. A large body of existing work focuses on practical implementation techniques to accelerate inference in deployed systems, often without accompanying theoretical guarantees. For instance, \citet{patel2023splitwise} propose system architectures that process prompt and token requests independently, while other works such as \citet{yu2022orca, agrawal2023sarathi, agrawal2024taming} introduce designs that jointly handle prompt and token requests---a structure aligned with the one studied in this paper. Our analytical model is narrower than these serving systems: it abstracts prefill computation into the initial KV-cache footprint and assigns unit time to each decoding slot.

In contrast, a few recent studies aim to develop theoretical foundations for LLM inference \cite{zhou2026position}. The primary mathematical model considered in this paper builds on the framework introduced by \citet{jaillet2025online}. \citet{ao2025optimizing} examine inference scheduling under multiple objectives and propose algorithms with provable performance guarantees. Additionally, \citet{chen2025adaptively,chen2026universal,bu2026tackling} study scheduling optimization with unknown decode lengths, and \citet{nie2026queueing} analyze stability conditions for LLM inference when operating on a single computational worker. These studies address settings with online arrivals, prediction uncertainty, or stability objectives, whereas our main guarantee concerns an offline backlog with known decode lengths and heterogeneous prefill sizes.

\xhdr{Scheduling. } The scheduling problem---extensively studied in the literature \citep{allahverdi2008survey, chen1998review, brucker1999resource, albers2009online}---involves a decision-maker tasked with processing a large number of incoming requests, either one-by-one or in batches. The challenge lies in determining the optimal order and timing of processing, which is often formulated as an integer optimization problem. However, due to the computational complexity and the impracticality of solving large-scale instances exactly, researchers have developed fast approximation algorithms and heuristics.

In our setting, which is motivated by LLM inference, jobs can be processed in batches. This aligns our problem more closely with batch-scheduling models studied in \citep{brucker1998scheduling, chen2008logistics, li2020online}. Furthermore, LLM inference imposes a sequential structure on token generation---future tokens cannot be processed before current ones---introducing a form of precedence constraint. Related work on scheduling with precedence constraints includes \citep{shahout2024don, schedPrecedence}. Our problem differs from these classical models because a request's resource requirement is not fixed: its KV-cache footprint grows with decoding age, and completed requests release memory continuously. The resulting tradeoff between batch cardinality, output length, and time-varying memory is the source of the new approximation analysis in this paper.

\section{Supplementary Materials for Section~\ref{sec:model}} \label{append:model}

\begin{proof}{Proof of Theorem~\ref{thm:benchmark}.}
For an integer $m\ge2$, let $M=3m^4$. Construct $m^4$ type-A requests
with $(s_i,o_i)=(3m^2-1,1)$ and $m^5$ type-B requests with
$(s_i,o_i)=(1,2)$. All requests are individually feasible because
$s_i+o_i\le M$. Since type A has output length $1$ and type B has output
length $2$, \mcsdf{} places all type-A requests before type B.

Exactly $m^2$ type-A requests fit when initiated together: in their only
decoding slot each has footprint $3m^2$, so their total footprint is
$3m^4=M$. Because all type-A requests have the same output length, a new
group cannot be admitted until the preceding group completes. Thus type A
is admitted in $m^2$ groups of at most $m^2$ requests, and the last group
starts at boundary $m^2-1$. The head-of-line rule prevents any type-B
request from starting earlier. Consequently,
\[
\mathrm{TEL}(\mcsdf;\mathcal I)
\ge m^5(m^2-1).
\]

For comparison, process type B first in $m$ batches of $m^4$ requests,
then type A in $m^2$ batches of $m^2$ requests. A type-B batch uses
$3m^4=M$ memory and lasts two slots; a type-A batch uses $3m^4=M$
memory and lasts one slot. The comparison schedule therefore finishes
in $2m+m^2$ slots. Charging all type-B requests at the end of the first
phase and all type-A requests at the end of the schedule gives
\[
\begin{aligned}
\mathrm{TEL}(\Lambda;\mathcal I)
&\le (m^5+m^4)(2m)+m^4m^2\\
&=3m^6+2m^5\le4m^6.
\end{aligned}
\]
Since $\mathrm{TEL}(\textup{\opt};\mathcal I)\le
\mathrm{TEL}(\Lambda;\mathcal I)$,
\[
\frac{\mathrm{TEL}(\mcsdf;\mathcal I)}
     {\mathrm{TEL}(\textup{\opt};\mathcal I)}
\ge\frac{m^5(m^2-1)}{4m^6}
\ge\frac m8\longrightarrow\infty,
\]
where the last inequality uses $m^2-1\ge m^2/2$ for $m\ge2$.
The construction also has
$\max_i(s_i+o_i)/M=3m^2/(3m^4)=m^{-2}\to0$.
\Halmos
\end{proof}

\begin{theorem} \label{thm:benchmark2}
    Suppose that each request $i \in [n]$ satisfies $s_i \in [1,M]$, $o_i \in [1,M]$ and $s_i+o_i \leq M$. Then \mcsdff{}, which orders requests by nondecreasing $s_i+o_i$ and retains the feasibility check and head-of-line admission rule of \mcsdf{}, has an unbounded approximation ratio. In particular, there exists a family of finite feasible instances $\{\mathcal I_m\}_{m\ge2}$ such that
    \[
    \frac{\mathrm{TEL}(\textup{\mcsdff};\mathcal I_m)}
         {\mathrm{TEL}(\textup{\opt};\mathcal I_m)}
    \longrightarrow\infty
    \qquad\text{as }m\to\infty.
    \]
\end{theorem}

\begin{proof}{Proof of Theorem~\ref{thm:benchmark2}.}
For an integer $m\ge2$, let $M=m^4$. Construct $m^4$ type-A requests
with $(s_i,o_i)=(1,m^2-1)$ and $m^5$ type-B requests with
$(s_i,o_i)=(m^2,1)$. Type A has peak size $m^2$, whereas type B has
peak size $m^2+1$, so \mcsdff{} places all type-A requests before type B.

Exactly $m^2$ type-A requests fit when initiated together. They occupy
all $M$ memory in their final decoding slot, and the equal-length groups
therefore start at boundaries $0,m^2-1,\ldots,(m^2-1)^2$. Since the
algorithm does not skip a blocked pending request, no type-B request starts
before the last of these boundaries. Hence
\[
    \mathrm{TEL}(\textup{\mcsdff};\mathcal I)
    \ge m^5(m^2-1)^2.
\]
This lower bound allows type-B requests to overlap the final type-A group.

Consider a comparison schedule $\Lambda$ that first processes type B in
batches of at most $m^2-1$ requests, then type A in batches of $m^2$.
These batches satisfy the static peak-memory constraint. The type-B phase
lasts
\[
    T_B=\left\lceil\frac{m^5}{m^2-1}\right\rceil\le2m^3,
\]
and the type-A phase lasts $m^2(m^2-1)$ slots. Charging every type-B
request at the end of the type-B phase and every type-A request at the end
of the complete schedule gives
\[
\begin{aligned}
    \mathrm{TEL}(\Lambda;\mathcal I)
    &\le(m^5+m^4)T_B+m^4m^2(m^2-1)\\
    &\le3m^8+2m^7\le4m^8.
\end{aligned}
\]
Since $(m^2-1)^2\ge m^4/4$ and the optimum is no more costly than
$\Lambda$,
\[
    \frac{\mathrm{TEL}(\textup{\mcsdff};\mathcal I)}
         {\mathrm{TEL}(\textup{\opt};\mathcal I)}
    \ge\frac{m^5(m^2-1)^2}{4m^8}
    \ge\frac{m}{16}\longrightarrow\infty.
\]
The construction also satisfies
$\max_i(s_i+o_i)/M=(m^2+1)/m^4\to0$.
\Halmos
\end{proof}

\begin{proof}{Proof of Theorem~\ref{thm:latencynp}.}
% The following reduction is retained from the original manuscript.
We prove the result via a reduction from the restricted 3-Partition problem. Recall that 3-Partition remains strongly NP-hard even when the input integers $x_i$ satisfy $T/4 < x_i < T/2$ for all $i$; see \cite{gareyjohnson1979}, Problem~SP15.

Consider the following 3-Partition problem: given a multiset of integers
$X=\{x_1,\ldots,x_{3m}\}$ summing to $mT$, where each $x_i$ satisfies
$T/4 < x_i < T/2$, the problem asks whether $X$ can be partitioned into
$m$ disjoint subsets $S_1,\ldots,S_m$ such that the sum of each subset equals $T$.

Then, we construct the following reduction: given an instance of 3-Partition, we construct an instance as follows:
\begin{itemize}
    \item For each integer $x_i$, create a request $i$ with $s_i=x_i$ and $o_i=1$.
    \item Set the memory capacity $M=T+3$.
\end{itemize}

This construction satisfies the conditions in Theorem~\ref{thm:latencynp}: for every request $i$, we have $s_i=x_i\in[1,M]$, $o_i=1\in[1,M]$, and
\[
s_i+o_i=x_i+1<T/2+1\le T+3=M.
\]

We show that the 3-Partition instance has a solution if and only if the constructed instance has a schedule with total end-to-end latency
\[
\mathrm{TEL}=\frac{3m(m+1)}2.
\]

\textit{Case 1: 3-Partition exists.}
Suppose $X$ can be partitioned into $m$ subsets $S_1,\ldots,S_m$, each summing to $T$. Then, consider the schedule which processes batch $k$ with all requests in $S_k$.
The schedule is feasible since each subset contains exactly three elements, and the memory usage of batch $k$ is
\[
\sum_{i\in S_k}(s_i+1)=\sum_{i\in S_k}(x_i+1)=T+3=M.
\]
Moreover, as each batch processes exactly three requests due to $T/4<x_i<T/2$, the total end-to-end latency is
\[
\mathrm{TEL}=\sum_{t=1}^m3t=\frac{3m(m+1)}2.
\]

\textit{Case 2: Converse.}
Suppose that the constructed scheduling instance admits a schedule
with total end-to-end latency
\[
\mathrm{TEL}=\frac{3m(m+1)}{2}.
\]
No feasible batch can contain four or more requests, because any four requests would require memory strictly greater than
\[
4\left(\frac T4+1\right)=T+4>T+3=M.
\]
Hence, each slot completes at most three requests. Since there are $3m$ requests, every schedule has total latency at least
\[
\sum_{t=1}^m3t=\frac{3m(m+1)}2,
\]
with equality only if exactly three requests complete in each of the first $m$ slots. Let $B_k$ be the three requests completed in slot $k$. Feasibility gives
\[
\sum_{i\in B_k}(x_i+1)\le T+3,
\]
so $\sum_{i\in B_k}x_i\le T$. Since these $m$ batches contain all $3m$ requests and $\sum_i x_i=mT$, every inequality is an equality. Thus $B_1,\ldots,B_m$ form a valid 3-Partition.

Therefore, the minimum TEL equals $3m(m+1)/2$ if and only if the 3-Partition instance is a yes-instance. Minimizing TEL is NP-hard.
\Halmos
\end{proof}

\subsection{Notation Table} \label{append:notation}

\begin{longtable}{p{0.22\textwidth}p{0.70\textwidth}}
\caption{Summary of Main Notation}\label{tab:notation}\\
\toprule
\textbf{Notation} & \textbf{Meaning} \\
\midrule
\endfirsthead
\toprule
\textbf{Notation} & \textbf{Meaning} \\
\midrule
\endhead

\multicolumn{2}{l}{\textit{Instance, requests, and schedules}}\\
$\mathcal I$ & Set of all requests in the instance. \\
$n$ & Number of requests, $n=|\mathcal I|$. \\
$r_i$ & Request $i$. \\
$a_i$ & Arrival time of request $i$; $a_i=0$ for all requests in the offline model. \\
$s_i$ & Prefill size, i.e., the number of input tokens of request $i$. \\
$o_i$ & Decode length, i.e., the number of output tokens generated by request $i$. \\
$M$ & KV-cache memory capacity of the worker. \\
$t$ & Integer boundary immediately before a decoding slot; $t=0$ is the initial boundary. \\
$p_i$ & Initiation boundary of request $i$; its tokens are generated in slots $p_i+1,\ldots,p_i+o_i$. \\
$c_i(\Lambda)$ & Completion time of request $i$ under schedule $\Lambda$. \\
$\Lambda$ & A generic schedule or scheduling policy. \\
$\textup{\opt}$ & An optimal offline schedule. \\
$\mathrm{TEL}(\Lambda;\mathcal I)$ & Total end-to-end latency of schedule $\Lambda$ on instance $\mathcal I$. \\
$\mathrm{AR}(\Lambda)$ & Approximation ratio of schedule $\Lambda$. \\

\midrule
\multicolumn{2}{l}{\textit{Batch execution and memory feasibility}}\\
$\mathcal B_t$ & Requests generating one token in decoding slot $t$. \\
$d_i(t)$ & Decoding age (output-token index) of request $i$ in slot $t$: $d_i(t)=t-p_i\in\{1,\ldots,o_i\}$ for $i\in\mathcal B_t$. \\
$\mathcal R^{(t)}$ & Pending requests at boundary $t$, before the admission decision. \\
$\mathcal V^{(t)}$ & Previously initiated, unfinished requests at boundary $t$. \\
$\mathcal U^{(t)}$ & Requests admitted at boundary $t$ and first decoded in slot $t+1$. \\
$U$ & Candidate subset of pending requests considered for admission. \\
$t_{\max}(U)$ & $t+\max_{i\in U}o_i$, the latest completion time if $U$ starts at boundary $t$. \\
$\mathcal I'$ & Ordered request list constructed by Phase~1 of \sortF{}. \\

\midrule
\multicolumn{2}{l}{\textit{\sortF{} metric and Phase~1 batches}}\\
$\mathcal X$ & Generic feasible request subset or candidate batch. \\
$F(\mathcal X)$ & Batch quality metric, $F(\mathcal X)=\sum_{r_i\in\mathcal X}o_i/|\mathcal X|^2$. \\
$\bar o(\mathcal X)$ & Average decode length of requests in $\mathcal X$. \\
$1/F(\mathcal X)$ & Throughput-density proxy $|\mathcal X|/\bar o(\mathcal X)$. \\
$\mathcal X_k$ & The $k$-th batch selected in Phase~1 of \sortF{}. \\
$n_k$ & Number of requests in $\mathcal X_k$. \\
$D_k$ & $\max_{r_i\in\mathcal X_k}o_i$, the maximum output length of batch $k$. \\
$f_k$ & $F(\mathcal X_k)$. \\
$\mathcal R_k$ & $\bigcup_{h=k}^K\mathcal X_h$, the residual set before batch $k$ is selected. \\
$N_k$ & $|\mathcal R_k|$, with $N_{K+1}=0$. \\
$T_k$ & $\sum_{h<k}D_h$, the serial start-boundary upper bound for batch $k$. \\
$H$ & $\sum_{k=1}^K N_kn_kf_k$, the proof charging quantity. \\
$D(\mathcal B)$ & $\max_{r_i\in\mathcal B}o_i$ for a nonempty batch $\mathcal B$. \\

\midrule
\multicolumn{2}{l}{\textit{\sortF{} proof quantities}}\\
$w_i$ & $s_i+o_i$, the static peak-memory size of request $i$. \\
$A(\mathcal S)$ & $\sum_{r_i\in\mathcal S}w_io_i$, the peak-area surrogate for a request subset. \\
$c_i^*$ & Completion time of request $i$ in an optimal schedule. \\
$q_k(t)$ & $|\{r_i\in\mathcal R_k:c_i^*\le t\}|$, requests from $\mathcal R_k$ completed by $t$. \\
$N^*(t)$ & $|\{r_i\in\mathcal I:c_i^*>t\}|$, requests unfinished after $t$. \\
$\mathcal Q_k$ & $[0,N_kf_k/10]\times(N_{k+1}/2,N_k/2]$, the lower-bound rectangle. \\

\midrule
\multicolumn{2}{l}{\textit{Exact, heuristic, and LP-based implementations}}\\
$dp[k][m]$ & Dynamic-programming state: minimum total decode length among $k$-request batches using memory $m$. \\
$ptr[k][m]$ & Immutable predecessor record used to reconstruct the DP batch. \\
$Q_p,Q_o$ & Quantile thresholds used in Quantile Greedy Selection. \\
$\bar T$ & Completion horizon with $\bar T\ge\sum_i o_i$ for an optimality-valid LP lower bound. \\
$\mathcal T_i$ & $\{0,\ldots,\bar T-o_i\}$, allowable initiation boundaries for request $i$. \\
$x_{i,t}$ & IP/LP variable indicating whether request $i$ starts at boundary $t$. \\
$x^*_{i,t}$ & Optimal fractional solution of the LP relaxation. \\
$y_i$ & $\sum_{t\in\mathcal T_i}t x^*_{i,t}$, the LP-induced mean initiation boundary. \\

\midrule
\multicolumn{2}{l}{\textit{Prediction uncertainty}}\\
$\ell_i,u_i$ & Lower and upper prediction bounds for the decode length of request $i$. \\
$\gamma$ & Prediction-confidence parameter, $\gamma:=\min_i \ell_i/u_i$ for the finite request set. \\
$\tilde{o}_i$ & Running estimate of the output length of request $i$. \\
$\mathcal A_{\min}$ & Adaptive refinement policy that starts from lower-bound predictions. \\
$\mathcal A_{\max}$ & Conservative policy that uses upper-bound predictions. \\

\midrule
\multicolumn{2}{l}{\textit{Online experiment}}\\
$\lambda$ & Poisson arrival rate in the preliminary online experiment. \\
$\lambda^*$ & Memory-bound throughput estimate used to normalize load. \\
$\rho$ & Normalized load, $\rho=\lambda/\lambda^*$. \\
$T$ & Arrival horizon in the online experiment. \\
\bottomrule
\end{longtable}

\section{Supplementary Materials for Section~\ref{sec:algorithm}}
\label{append:alg}

\begin{proof}{Proof of Lemma~\ref{lemma:output_balance}.}
Let $b=|\mathcal X|$ and $S=\sum_{r_i\in\mathcal X}o_i$.
If $b=1$, then $\bar o=o_{\max}$ and the assertion follows.
For $b\ge2$, remove a request of length $o_{\max}$ to obtain
$\mathcal X^-$. Both Phase~1 candidate families are closed under taking
nonempty subsets. Thus $\mathcal X^-$ is an admissible candidate, and
exact minimization gives
\[
\frac{S}{b^2}\le\frac{S-o_{\max}}{(b-1)^2}.
\]
Rearranging yields $(2b-1)S\ge b^2o_{\max}$, so
\[
\bar o=\frac Sb\ge\frac{b}{2b-1}o_{\max}>\frac12o_{\max}.
\]
\Halmos
\end{proof}

\subsection{Example for Scheduling Intuition}
\label{append:example}
\begin{example}\label{exp2}
Consider a simple scenario with two batches of requests where requests are
homogeneous within each batch but heterogeneous across batches, as shown
in Table~\ref{tab:batch_characteristics}.
\begin{table}[htbp]
\centering
\caption{Batch Characteristics and Quality Metrics}
\label{tab:batch_characteristics}
\begin{tabular}{ccccc}
\hline
Batch & Prompt Size & Response Length & Requests Number & Quality Metric\\
\hline
1 & $s_1$ & $o_1$ & $n_1$ & $o_1/n_1$\\
2 & $s_2$ & $o_2$ & $n_2$ & $o_2/n_2$\\
\hline
\end{tabular}
\end{table}
Assume that both batches have the same total memory requirement:
$n_1(s_1+o_1)=n_2(s_2+o_2)=M$.
We compare the two schedules that keep these batches
intact and execute them serially.

If batch 1 is prioritized, it completes at time $o_1$, contributing
$o_1n_1$ to the total end-to-end latency (TEL). Batch 2 then completes at
time $o_1+o_2$, contributing $o_2n_2+o_1n_2$ to TEL. Thus, the TEL is
$o_1n_1+o_2n_2+o_1n_2$.
Conversely, if batch 2 is prioritized, it completes at time $o_2$,
contributing $o_2n_2$ to TEL. Batch 1 then completes at time $o_1+o_2$,
contributing $o_1n_1+o_2n_1$ to TEL. The TEL in this case is
$o_1n_1+o_2n_2+o_2n_1$.

The first order is no worse precisely when
$o_1/n_1\le o_2/n_2$, which is the order prescribed by $F$ for these
homogeneous batches. This comparison concerns the two fixed serial
orders; it does not assert optimality over schedules that split or
overlap the batches.
\end{example}

\subsection{Pseudocode of \sortF}
For a residual set $\mathcal R$, let $\mathcal F(\mathcal R)$
be either of the two nonempty candidate families in Section~\ref{sec:algorithm}.
The choice is fixed for the analyzed version. All ties not otherwise
specified may be broken by request index. The memory function is defined
in \eqref{eq:sf_memory_profile}.

\begin{algorithm}[htbp]
\caption{\sortF: exact batch selection and ordered admission}
\label{algorithm_1}
\begin{algorithmic}[1]
\Require Feasible requests $\mathcal I$ with known lengths; capacity $M$; candidate family $\mathcal F$
\Ensure A feasible schedule with initiation times $p_i$ and completion times $c_i$
\State \textbf{Phase 1:} $\mathcal R\gets\mathcal I$; $\mathcal I'\gets[\,]$
\While{$\mathcal R\ne\varnothing$}
  \State $\mathcal X\gets\arg\min_{\mathcal X\in\mathcal F(\mathcal R)}(F(\mathcal X),-|\mathcal X|)$ lexicographically
  \State Append $\mathcal X$, ordered by nondecreasing $o_i$, to $\mathcal I'$
  \State $\mathcal R\gets\mathcal R\setminus\mathcal X$
\EndWhile
\State \textbf{Phase 2:} $t\gets0$; pending list $\mathcal R\gets\mathcal I'$; $\mathcal V\gets\varnothing$
\While{$\mathcal R\ne[\,]$ or $\mathcal V\ne\varnothing$}
  \State Remove from $\mathcal V$ all requests with $p_i+o_i\le t$
  \While{$\mathcal R\ne[\,]$}
    \State Let $i$ be the first pending request; tentatively set $p_i\gets t$ and $\mathcal W\gets\mathcal V\cup\{i\}$
    \If{$M(\mathcal W,p_j+o_j)\le M$ for every $j\in\mathcal W$}
      \State Remove $i$ from the head of $\mathcal R$; $\mathcal V\gets\mathcal W$
    \Else
      \State Discard the tentative $p_i$; \textbf{break} the admission loop
    \EndIf
  \EndWhile
  \State Each request in $\mathcal V$ generates one token in slot $t+1$; $t\gets t+1$
  \State Set $c_i\gets t$ for requests that finish in this slot and remove them from $\mathcal V$
\EndWhile
\end{algorithmic}
\end{algorithm}

All requests considered by an admission check have already
started or are tentatively started at the current boundary. Between two
future completion slots their combined footprint is nondecreasing;
therefore checking these completion slots suffices. If the active set is
empty, the first pending request always fits by individual feasibility,
so Phase~2 cannot deadlock. Replacing Phase~1's exact selection with a
heuristic preserves this admission rule and its feasibility, but not the
approximation guarantee.

\section{Supplementary Materials for Section~\ref{sec:proof}}
\label{append:proof}

We use the notation of Section~\ref{sec:proof}. For a nonempty request set
$\mathcal B$, write $D(\mathcal B)=\max_{r_i\in\mathcal B}o_i$.
Both Phase~1 candidate families in Theorem~\ref{thm:cr} are closed under
taking nonempty subsets, consist of synchronously feasible batches, and
contain every batch satisfying
$\sum_{r_i\in\mathcal B}(s_i+o_i)\le M$.

\begin{proof}{Proof of Lemma~\ref{lemma:sf_serial}.}
We first establish that ordered, staggered initiation preserves feasibility
within a batch. Let $\mathcal X$ be synchronously feasible, and order its
requests so that $o_1\le\cdots\le o_b$. Assign any nonnegative integer initiation
times in nondecreasing order, $p_1\le\cdots\le p_b$. At a slot $u$ with active requests, let $i$ be
the first active request in this order and put $a=u-p_i$. Every active
request $j$ satisfies
\[
    p_j\ge p_i,\qquad o_j\ge o_i\ge a,
    \qquad 1\le u-p_j\le a.
\]
Hence its batch's total memory usage is bounded by
\begin{equation}
\label{eq:sf_staggered_feasibility}
    \sum_{j:\,p_j<u\le p_j+o_j}(s_j+u-p_j)
    \le\sum_{j:\,o_j\ge a}(s_j+a)
    \le M.
\end{equation}
The last inequality is synchronous feasibility at age $a$.

We next prove by induction that every request in $\mathcal X_k$ starts by
$T_k=\sum_{h<k}D_h$. The first batch can be admitted in full at time zero.
Suppose the claim holds for every $h<k$. Every request in such a batch
completes by $T_h+D_h=T_{h+1}\le T_k$, so all earlier batches have completed
by $T_k$. If $\mathcal X_k$ has not yet been fully admitted, no later batch
has started, because Phase~2 does not skip a blocked request. Initiating all
remaining requests of $\mathcal X_k$ at $T_k$ preserves its nondecreasing
initiation order and is feasible by \eqref{eq:sf_staggered_feasibility}.
Every prefix of this admission is also feasible. Thus Phase~2 admits all
remaining requests of $\mathcal X_k$ no later than $T_k$, completing the
induction.

It follows that every request in $\mathcal X_k$ completes by $T_k+D_k$.
Since $N_k=\sum_{h=k}^Kn_h$,
\[
    \mathrm{TEL}(\textup{\sortF};\mathcal I)
    \le\sum_{k=1}^K n_k(T_k+D_k)
    =\sum_{k=1}^K N_kD_k.
\]
Lemma~\ref{lemma:output_balance} applies to either candidate family, since
its proof uses only exact $F$-minimization and closure under taking subsets.
It gives $D_k<2n_kf_k$, including singleton batches, and therefore
$\sum_kN_kD_k\le2H$.
\Halmos
\end{proof}

\begin{proof}{Proof of Lemma~\ref{lemma:sf_deadline}.}
The empty set is immediate. Fix a nonempty $\mathcal S$ as in the lemma,
write $w_i=s_i+o_i$, and define
\[
    A(\mathcal S)=\sum_{r_i\in\mathcal S}w_io_i.
\]
All tokens of these requests are generated by time $t$. Summing the memory
constraints over the first $\lfloor t\rfloor$ slots gives
\begin{equation}
\label{eq:sf_memory_time}
    \sum_{r_i\in\mathcal S}
    \left(s_io_i+\frac{o_i(o_i+1)}2\right)
    \le M\lfloor t\rfloor\le Mt.
\end{equation}
Since
\[
    w_io_i=s_io_i+o_i^2
    \le2\left(s_io_i+\frac{o_i(o_i+1)}2\right),
\]
we obtain
\begin{equation}
\label{eq:sf_peak_area}
    A(\mathcal S)\le2Mt,
    \qquad D(\mathcal S)\le t.
\end{equation}
The second inequality follows because a request needs $o_i$ distinct
decoding slots and all requests arrive at time zero.

Order $\mathcal S$ by nonincreasing $o_i$. Starting with an empty batch,
append requests in this order until adding the next request would make the
sum of $w_i$ exceed $M$; then open a new batch. Every request fits in an
empty batch because $w_i\le M$. This procedure yields a partition
$\mathcal B_1,\ldots,\mathcal B_m$ satisfying the static memory constraint.
For each batch, let
\[
    h_j=D(\mathcal B_j),\qquad
    W_j=\sum_{r_i\in\mathcal B_j}w_i,\qquad
    A_j=\sum_{r_i\in\mathcal B_j}w_io_i,
\]
and let $r_{v_j}$ be its first request. For $j\ge2$, opening batch $j$
implies $W_{j-1}+w_{v_j}>M$. Every request in $\mathcal B_{j-1}$ has output
length at least $h_j=o_{v_j}$. Consequently,
\[
    Mh_j<(W_{j-1}+w_{v_j})h_j
    \le A_{j-1}+w_{v_j}o_{v_j}.
\]
Summing these inequalities, with empty sums interpreted as zero, yields
\[
\begin{aligned}
    \sum_{j=1}^m h_j
    &\le h_1+\frac1M
       \left(\sum_{j=1}^{m-1}A_j
             +\sum_{j=2}^m w_{v_j}o_{v_j}\right)\\
    &\le h_1+\frac{2A(\mathcal S)}M
    \le5t,
\end{aligned}
\]
where the last inequality uses \eqref{eq:sf_peak_area}. This proves
\eqref{eq:sf_deadline_bound}.
\Halmos
\end{proof}

\begin{proof}{Proof of Lemma~\ref{lemma:sf_lower_bound}.}
Fix $k$ and $t>0$. Apply Lemma~\ref{lemma:sf_deadline} to
\[
    \mathcal S_k(t)=\{r_i\in\mathcal R_k:c_i^*\le t\}.
\]
If this set is nonempty, let $\mathcal B_1,\ldots,\mathcal B_m$ be the
resulting static-feasible partition. Each $\mathcal B_j$ is a subset of
$\mathcal R_k$ and a feasible candidate under either Phase~1 selection
rule. Exact $F$-minimization therefore gives
\[
    f_k\le F(\mathcal B_j)
    =\frac{\sum_{r_i\in\mathcal B_j}o_i}{|\mathcal B_j|^2}
    \le\frac{D(\mathcal B_j)}{|\mathcal B_j|}.
\]
Multiplying by $|\mathcal B_j|$ and summing over the partition, we obtain
\[
    q_k(t)f_k
    \le\sum_{j=1}^mD(\mathcal B_j)\le5t.
\]
The same inequality is immediate if $\mathcal S_k(t)$ is empty or $t=0$.

Let $N^*(t)=|\{r_i\in\mathcal I:c_i^*>t\}|$. At least
$N_k-q_k(t)$ requests of $\mathcal R_k$ are unfinished in the optimal
schedule. Since $f_k>0$,
\begin{equation}
\label{eq:sf_unfinished_bound}
    N^*(t)\ge N_k-\frac{5t}{f_k},
    \qquad k=1,\ldots,K,\quad t\ge0.
\end{equation}
Also,
\begin{equation}
\label{eq:sf_tail_integral}
    \mathrm{TEL}(\textup{\opt};\mathcal I)
    =\int_0^\infty N^*(t)\,dt.
\end{equation}

For each $k$, consider the rectangle
\[
    \mathcal Q_k=
    \left[0,\frac{N_kf_k}{10}\right]
    \times\left(\frac{N_{k+1}}2,\frac{N_k}2\right]
\]
in the time--unfinished-request plane. If $t\le N_kf_k/10$,
\eqref{eq:sf_unfinished_bound} gives $N^*(t)\ge N_k/2$. Thus
$\mathcal Q_k$ lies below the graph of $N^*$. These rectangles have
disjoint vertical interiors, because $N_k-N_{k+1}=n_k$. No ordering of the widths $N_kf_k/10$ is required. Their total area
is therefore at most the integral in \eqref{eq:sf_tail_integral}, so
\[
    \mathrm{TEL}(\textup{\opt};\mathcal I)
    \ge\sum_{k=1}^K\frac{n_k}{2}\frac{N_kf_k}{10}
    =\frac{H}{20}.
\]
This proves \eqref{eq:sf_optimal_bound}.
\Halmos
\end{proof}

\section{Supplementary Materials for Section~\ref{sec:approximation}}
\label{append:approx}

\subsection{Exact Static Batch Selection}
In Algorithm~\ref{alg:exact_dp}, a predecessor record is an
immutable pair consisting of one request index and a pointer to an older
record. Creating a record does not copy a request list. This detail is
needed for the stated reconstruction complexity.

\begin{algorithm}[htbp]
\caption{Exact static batch selection by dynamic programming}
\label{alg:exact_dp}
\begin{algorithmic}[1]
\Require Nonempty residual set $\mathcal R$ of $r$ requests; integer capacity $M$
\Ensure A minimizer of $(F(\mathcal X),-|\mathcal X|)$ over static-feasible nonempty subsets
\State Relabel the residual requests as $1,\ldots,r$; set $w_i\gets s_i+o_i$
\State Initialize $dp[k][m]\gets+\infty$, $ptr[k][m]\gets\mathrm{null}$ for $0\le k\le r$, $0\le m\le M$
\State $dp[0][0]\gets0$
\For{$i=1,\ldots,r$}
  \For{$k=i-1,i-2,\ldots,0$}
    \For{$m=0,\ldots,M-w_i$}
      \If{$dp[k][m]<+\infty$ and $dp[k][m]+o_i<dp[k+1][m+w_i]$}
        \State $dp[k+1][m+w_i]\gets dp[k][m]+o_i$
        \State $ptr[k+1][m+w_i]\gets\operatorname{NewNode}(i,ptr[k][m])$ \Comment{immutable}
      \EndIf
    \EndFor
  \EndFor
\EndFor
\State Among finite states with $k\ge1$, select $(k^*,m^*)$ minimizing $(dp[k][m]/k^2,-k)$
\State Follow $ptr[k^*][m^*]$ to reconstruct and \Return $\mathcal X^*$
\end{algorithmic}
\end{algorithm}

After processing request $i$, $dp[k][m]$ is the least
output sum among subsets of the first $i$ requests with cardinality $k$
and peak-memory sum $m$. The update either excludes or includes request
$i$, and descending $k$ ensures its inclusion at most once. This proves
the invariant by induction and the exactness of the final minimization.
Each predecessor record points to a previously constructed subset, so
later state updates do not change its contents. There are $O(r^2M)$
updates and records, giving $O(r^2M)$ time and total storage; the scalar
array alone occupies $O(rM)$. Objective ratios can be compared by cross
multiplication to avoid floating-point tie errors.

\subsection{Local Swap Search}
\begin{algorithm}[htbp]
\caption{Local Swap Search, optionally from a supplied feasible seed}
\label{alg:local_swap}
\begin{algorithmic}[1]
\Require Residual set $\mathcal R$; capacity $M$; optional nonempty static-feasible seed $\mathcal X_0$
\Ensure A static-feasible batch with no improving one-for-one swap
\State Set $w_i\gets s_i+o_i$ for all $i\in\mathcal R$
\If{a seed is supplied}
  \State $\mathcal X\gets\mathcal X_0$
\Else
  \State $\mathcal X\gets\varnothing$; $W\gets0$
  \For{each $i\in\mathcal R$ in nondecreasing $w_i$}
    \If{$W+w_i\le M$}
      \State $\mathcal X\gets\mathcal X\cup\{i\}$; $W\gets W+w_i$
    \EndIf
  \EndFor
\EndIf
\State $W\gets\sum_{i\in\mathcal X}w_i$; $S\gets\sum_{i\in\mathcal X}o_i$
\Loop
  \State Scan pairs $(a,b)\in\mathcal X\times(\mathcal R\setminus\mathcal X)$ in a fixed order
  \State Find the first pair satisfying $W-w_a+w_b\le M$ and $o_b<o_a$
  \If{no such pair exists}
    \State \Return $\mathcal X$
  \EndIf
  \State $\mathcal X\gets(\mathcal X\setminus\{a\})\cup\{b\}$; $W\gets W-w_a+w_b$; $S\gets S-o_a+o_b$
\EndLoop
\end{algorithmic}
\end{algorithm}

Let $b=|\mathcal X|$ and $r=|\mathcal R|$. Cardinality is
fixed during the swap phase, so an improving swap is exactly one that
reduces the output sum. Rank requests by nondecreasing $o_i$, breaking
ties by index. Each improving swap strictly decreases the sum of the
selected ranks. The largest possible difference between two sums of $b$
distinct ranks in $\{1,\ldots,r\}$ is $b(r-b)$. Hence there are at most
$b(r-b)$ successful swaps and $P\le b(r-b)+1$ scans, including the final
unsuccessful scan. The cost per call is $O(r\log r+Pr^2)$, with the coarse
worst-case bound $O(r^4)$. This termination argument does not bound the
quality of the local optimum.

\subsection{Quantile Greedy Selection}
\begin{algorithm}[htbp]
\caption{Quantile Greedy Selection}
\label{alg:quantile_greedy}
\begin{algorithmic}[1]
\Require Nonempty residual set $\mathcal R$; capacity $M$; a fixed empirical-quantile convention
\Ensure A nonempty static-feasible batch $\mathcal X$
\State Set $w_i\gets s_i+o_i$; sort $\mathcal R$ by nondecreasing $o_i$
\State Draw a uniform sample without replacement of size $h=\max\{1,\lfloor|\mathcal R|/2\rfloor\}$
\State Set $Q_p$ and $Q_o$ to the sample $0.3$-quantiles of $w_i$ and $o_i$ under the fixed convention
\State $\mathcal X\gets\varnothing$; $W\gets0$
\For{each $i\in\mathcal R$ in nondecreasing $o_i$}
  \If{$w_i\le Q_p$, $o_i\le Q_o$, and $W+w_i\le M$}
    \State $\mathcal X\gets\mathcal X\cup\{i\}$; $W\gets W+w_i$
  \EndIf
\EndFor
\For{each $i\in\mathcal R\setminus\mathcal X$ in nondecreasing $o_i/w_i$}
  \If{$W+w_i\le M$}
    \State $\mathcal X\gets\mathcal X\cup\{i\}$; $W\gets W+w_i$
  \EndIf
\EndFor
\State \Return $\mathcal X$
\end{algorithmic}
\end{algorithm}

Every insertion updates $W$, so the returned set is
static-feasible. If the core scan selects nothing, the first request in
the filling scan fits by individual feasibility. The two sorts and sample
quantiles can be computed in $O(r\log r)$ time with $O(r)$ storage.

\section{Supplementary Materials of Section \ref{sec:extension}} \label{append:extension}

\subsection{Completion Horizon and Pseudocodes for LP-Based Methods} \label{append:extension_pseudo}

\begin{lemma}[Sufficient completion horizon]
For every finite feasible instance, there exists an optimal schedule whose makespan is at most $\sum_{i=1}^n o_i$. Consequently, every integer horizon $\bar T\ge \sum_{i=1}^n o_i$ contains an optimal solution of the time-indexed IP.
\end{lemma}

\begin{proof}
Consider an optimal schedule and suppose that there is a globally idle
decoding slot $\tau$ before the final completion, i.e., no request
generates a token in slot $\tau$.

Because decoding is uninterrupted once a request starts, every request
with $p_i<\tau$ must have completed before slot $\tau$. Now shift every
request with $p_i\ge \tau$ one boundary earlier, replacing its initiation
time by $p_i-1$.

Among the shifted requests, relative initiation times and decoding ages
are unchanged. Moreover, the shifted requests do not overlap any request
initiated before $\tau$, since all such requests have already completed
before the idle slot. Hence all memory constraints remain feasible.

Because the idle slot occurs before the final completion, at least one
request is shifted, and every shifted completion time decreases by one.
The resulting schedule therefore has strictly smaller total completion
time, contradicting optimality. Thus an optimal schedule has no globally
idle decoding slot before its final completion.

Every occupied slot generates at least one output token, while the total
number of generated output tokens is $\sum_{i=1}^n o_i$. Hence the
makespan of such an optimal schedule is at most $\sum_{i=1}^n o_i$.
\end{proof}

\begin{algorithm}[htbp]
\caption{\sortLP} \label{algorithm_sortLP}
\begin{algorithmic}[1]
\State Solve \text{Relaxed-LP} and obtain an optimal fractional solution $x^*_{i,t}$ for all $i\in[n]$ and $t\in T_i$
\State Set $y_i \leftarrow \sum_{t\in T_i} t x^*_{i,t}$ for all $i\in[n]$
\State Sort all requests $r_i \in \mathcal{I}$ in ascending order of $y_i$ to obtain an ordered list $\mathcal{I}'$
\State Execute the ordered list $\mathcal I'$ using Phase~2 (lines 7--20) of Algorithm~1.
\end{algorithmic}
\end{algorithm}

\begin{algorithm}[htbp]
\caption{\LPSwap} \label{algorithm_LPswap}
\begin{algorithmic}[1]
\State Solve \text{Relaxed-LP} and obtain an optimal fractional solution $x^*_{i,t}$ for all $i\in[n]$ and $t\in T_i$
\State Set $y_i \leftarrow \sum_{t\in T_i} t x^*_{i,t}$ for all $i\in[n]$
\State $\mathcal{I}' \gets []$
\While{$\mathcal{I} \neq \emptyset$}
    \State Sort $\mathcal{I}$ by increasing $y_i$
    \State $\mathcal{X}_{\mathrm{init}}\gets\emptyset,\ mem\gets 0$
    \For{each $r\in\mathcal{I}$ in this order}
        \If{$mem+s_r+o_r\le M$} \State $\mathcal{X}_{\mathrm{init}}\gets\mathcal{X}_{\mathrm{init}}\cup\{r\}$; $mem\gets mem+s_r+o_r$ \EndIf
    \EndFor
    \State $\mathcal{X}^* \gets \textsc{LocalSwap}(\mathcal{X}_{\mathrm{init}}, \mathcal{I}, M)$
    \State Sort $\mathcal{X}^*$ by nondecreasing $o_i$
    \State Append $\mathcal{X}^*$ to $\mathcal{I}'$
    \State $\mathcal{I}\gets \mathcal{I}\setminus \mathcal{X}^*$
\EndWhile
\State Execute the ordered list $\mathcal I'$ using Phase~2 (lines 7--20) of Algorithm~1.
\end{algorithmic}
\end{algorithm}

\subsection{Simulation Results and Analysis}
\label{append:simulation}

We compare \sortLP{}, \sortF{} with Local Swap Search, and
\LPSwap{} on synthetic workloads generated from five distributions:
Uniform, Normal, Binomial, Exponential, and Mixed. These experiments
evaluate the implemented heuristics under the unit-time decoding
model and complement the real-workload evaluation in
Section~\ref{sec:num}.

\subsection*{Experimental Setup and Parameters}

We fix the memory capacity at $M=100$ and the maximum input and
output lengths at $s_{\max}=o_{\max}=50$. Each instance contains
$n=100$ requests for the Uniform, Normal, Binomial, and Exponential
distributions, and $n=200$ requests for the Mixed distribution.
Input and output lengths are generated independently, with the
following distributional specifications:
\begin{itemize}
    \item \textbf{Uniform:}
    $s_i$ and $o_i$ are sampled independently and uniformly from
    $\{1,\ldots,50\}$.

    \item \textbf{Normal:}
    Input and output lengths are generated from
    $\mathcal N(25,8.33^2)$ and clipped to $[1,50]$.

    \item \textbf{Binomial:}
    $s_i,o_i\sim\mathrm{Binomial}(49,0.5)+1$.

    \item \textbf{Exponential:}
    Input and output lengths are generated from
    $\mathrm{Exp}(\lambda=0.2)$, with scale $5$, and clipped
    to $[1,50]$.

    \item \textbf{Mixed:}
    To represent a workload with predominantly short inputs and
    a smaller proportion of longer inputs, we generate exactly
    $80\%$ of input lengths from an exponential distribution
    with rate $0.1$ and $20\%$ from a lognormal distribution
    with parameters $\mu=\ln 40$ and $\sigma=0.25$.
    Input values exceeding $50$ are replaced by independent
    draws from $\mathcal U(40,50)$, after which the inputs are
    clipped to $[1,50]$ and randomly permuted.
    Output lengths are generated independently from
    $\mathrm{Exp}(\lambda=0.2)$ and clipped to $[1,50]$.
\end{itemize}

For all continuous distributions, the resulting lengths are rounded
to the nearest integer, with half-integers rounded upward.
Figure~\ref{fig:mixed_hist} illustrates the Mixed input-length
distribution.

\begin{figure}[h]
\centering
\includegraphics[width=0.75\textwidth]{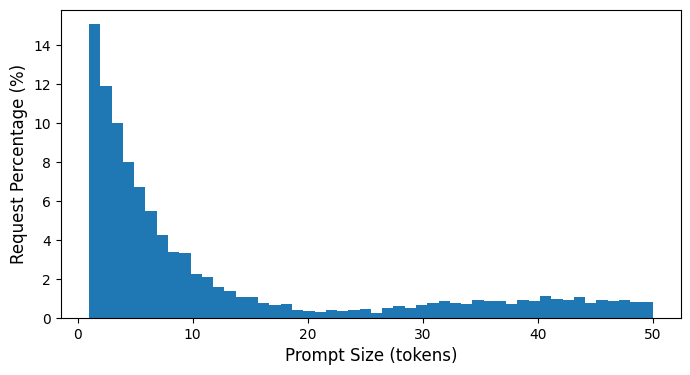}
\caption{Illustration of the input-length distribution in the
Mixed workload configuration.}
\label{fig:mixed_hist}
\end{figure}

\xhdr{Timing convention and replication.}
All requests are available at time zero. Time is measured in
decoding slots: a request initiated at boundary $p_i$ generates
one token in each slot $p_i+1,\ldots,p_i+o_i$ and completes at
$c_i=p_i+o_i$. Its KV-cache footprint in an active slot $\tau$
is $s_i+\tau-p_i$, including its final decoding slot.
We report total end-to-end latency,
$\mathrm{TEL}=\sum_i c_i$, rather than latency averaged over
requests. These values are measured in decoding slots, not
Vidur wall-clock seconds.

Each distribution is evaluated over five independent random seeds,
$0,1,2,3,4$. All three policies use the same realized requests
for each distribution and seed. We report the mean TEL and its
sample standard deviation across these five replications.

\xhdr{Algorithm implementation.}
We implement Local Swap Search as specified in Algorithm~3,
\sortLP{} as specified in Algorithm~\ref{algorithm_sortLP},
and \LPSwap{} as specified in Algorithm~\ref{algorithm_LPswap}.
All three methods execute their resulting request orders using
the same ordered-admission procedure in Phase~2 of Algorithm~1.
Ties are broken by the original request index, and Local Swap
Search selects the first improving pair in lexicographic
request-index order.

For each instance, the LP is solved using \texttt{SciPy/HiGHS},
with completion horizon $\bar T=\sum_i o_i$ and allowable start
boundaries $\{0,\ldots,\bar T-o_i\}$ for request $i$.
The same LP solution supplies the expected start times used by
\sortLP{} and \LPSwap{}. The LP is solved to the configured
numerical tolerances, and every resulting schedule is independently
checked against the dynamic memory constraints.

As an additional consistency check, we verify the necessary bound
\[
\mathrm{TEL}\ge
\frac{1}{M}\sum_{j=1}^{n}(n-j+1)v_{(j)},
\qquad
v_i=s_io_i+\frac{o_i(o_i+1)}{2},
\]
where $v_{(1)}\le\cdots\le v_{(n)}$ are the sorted
memory--time areas of the requests.
To see this, let $c_{(k)}$ denote the $k$th smallest completion
time. By time $c_{(k)}$, the $k$ completed requests have consumed
at least $\sum_{j=1}^{k}v_{(j)}$ units of memory--time area,
whereas the capacity constraint permits at most $Mc_{(k)}$.
Summing this inequality over $k$ yields the stated bound.

\subsection*{Numerical Results}

Table~\ref{tab:sim_results} reports the mean total end-to-end
latency across five seeds, with sample standard deviations in
parentheses. Lower values indicate better performance.
Comparisons between policies are made within each distribution,
using identical request instances.

\begin{table}[htbp]
\centering
\caption{Mean total end-to-end latency in decoding slots across
five seeds, with sample standard deviations in parentheses.
All policies use identical requests within each replication and
memory capacity $M=100$.}
\label{tab:sim_results}
\small
\begin{tabular}{lrrrc}
\toprule
Distribution & \sortLP{} & \sortF{} (Swap) & \LPSwap{} & $n$ \\
\midrule
Uniform
& 50112.0 (4851.0)
& 52256.8 (5362.8)
& 52569.0 (5341.0)
& 100 \\
Normal
& 51233.0 (2802.9)
& 53431.0 (3489.8)
& 53221.0 (3312.9)
& 100 \\
Binomial
& 60883.4 (616.7)
& 61126.8 (751.2)
& 61025.4 (735.0)
& 100 \\
Exponential
& 1538.6 (174.8)
& 1554.8 (163.3)
& 1576.2 (153.2)
& 100 \\
Mixed
& 10125.8 (851.3)
& 10609.8 (808.2)
& 10186.6 (781.6)
& 200 \\
\bottomrule
\end{tabular}
\end{table}

\subsection*{Key Observations}

The results support three descriptive observations.

\begin{enumerate}
    \item \textbf{Sorted-LP achieves the lowest mean TEL in
    these synthetic experiments.}
    \sortLP{} has the lowest reported mean TEL under all five
    distributions. Its advantage over \sortF{} with Local Swap
    Search is more pronounced under the Uniform, Normal, and
    Mixed distributions than under the Binomial and Exponential
    distributions. Thus, LP-guided ordering is an effective
    heuristic for the synthetic instances considered here.

    \item \textbf{The effect of LP initialization on swap-based
    batch construction depends on the workload.}
    Relative to \sortF{} with Local Swap Search, \LPSwap{}
    achieves a lower mean TEL under the Normal, Binomial, and
    Mixed distributions, but a higher mean TEL under the Uniform
    and Exponential distributions. In particular, under the Mixed
    distribution, LP initialization reduces the mean TEL from
    $10609.8$ to $10186.6$, although this remains above the
    $10125.8$ achieved by \sortLP{} alone.

    \item \textbf{Combining LP guidance with local swaps does not
    guarantee an improvement over LP ordering alone.}
    \LPSwap{} has a higher mean TEL than \sortLP{} under every
    tested distribution. This is consistent with the distinction
    between the local batch objective and the total latency of
    the executed schedule: an improving swap reduces the
    selected batch's output sum at fixed cardinality, but does
    not necessarily reduce the TEL of the complete schedule.
    Moreover, \LPSwap{} reconstructs batches from the LP-guided
    order rather than directly optimizing the TEL of the
    \sortLP{} schedule.
\end{enumerate}

These comparisons describe the observed sample means.
The reported standard deviations summarize variability across
instances; they do not, by themselves, establish statistical
significance of the paired policy differences.

The findings also illustrate the scope of the theoretical results.
The factor-$40$ guarantee requires exact Phase~1 selection and
does not imply that the Local Swap Search implementation must
outperform LP-guided heuristics. Similarly, these synthetic
unit-time results do not determine the ranking of policies under
the real-workload Vidur timing model. Section~\ref{sec:num}
examines performance in that separate setting.

\section{Supplementary Materials of Section \ref{sec:num}} \label{append:num}

Figure~\ref{fig:distribution3} reports the token-length distributions of the mixed workload used in the main real-data experiments. The figure highlights the heterogeneity created by combining short conversational requests with long-document summarization tasks: most prompts are short, while a smaller number of long-context requests creates a pronounced right tail. This mixed distribution is the regime in which memory-aware scheduling is especially important.

\begin{figure}[!h]
\center
\includegraphics[width=0.8\textwidth]{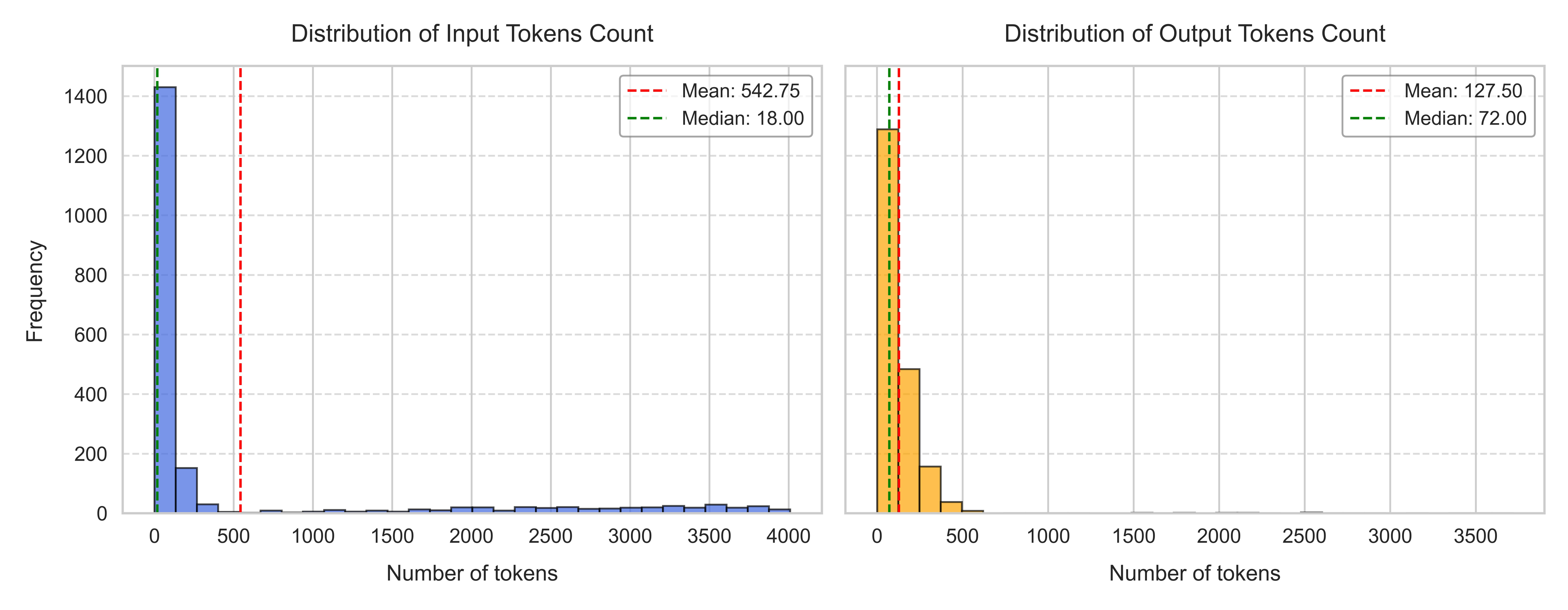}
\caption{Distribution of the number of tokens of input prompt and output response respectively in the mixed data} 
\label{fig:distribution3}
\end{figure}

Figure~\ref{fig:resultappend} provides a zoomed-in comparison between \sortF{} with Local Swap Search and \LPSwap{}, whose curves nearly overlap in Figure~3. The refined view shows that the two methods achieve very similar average latency across workload sizes, suggesting that the F-metric-guided local refinement captures most of the performance benefit obtained from LP-guided initialization.

\begin{figure}[!h]
\center
\includegraphics[width=0.75\textwidth]{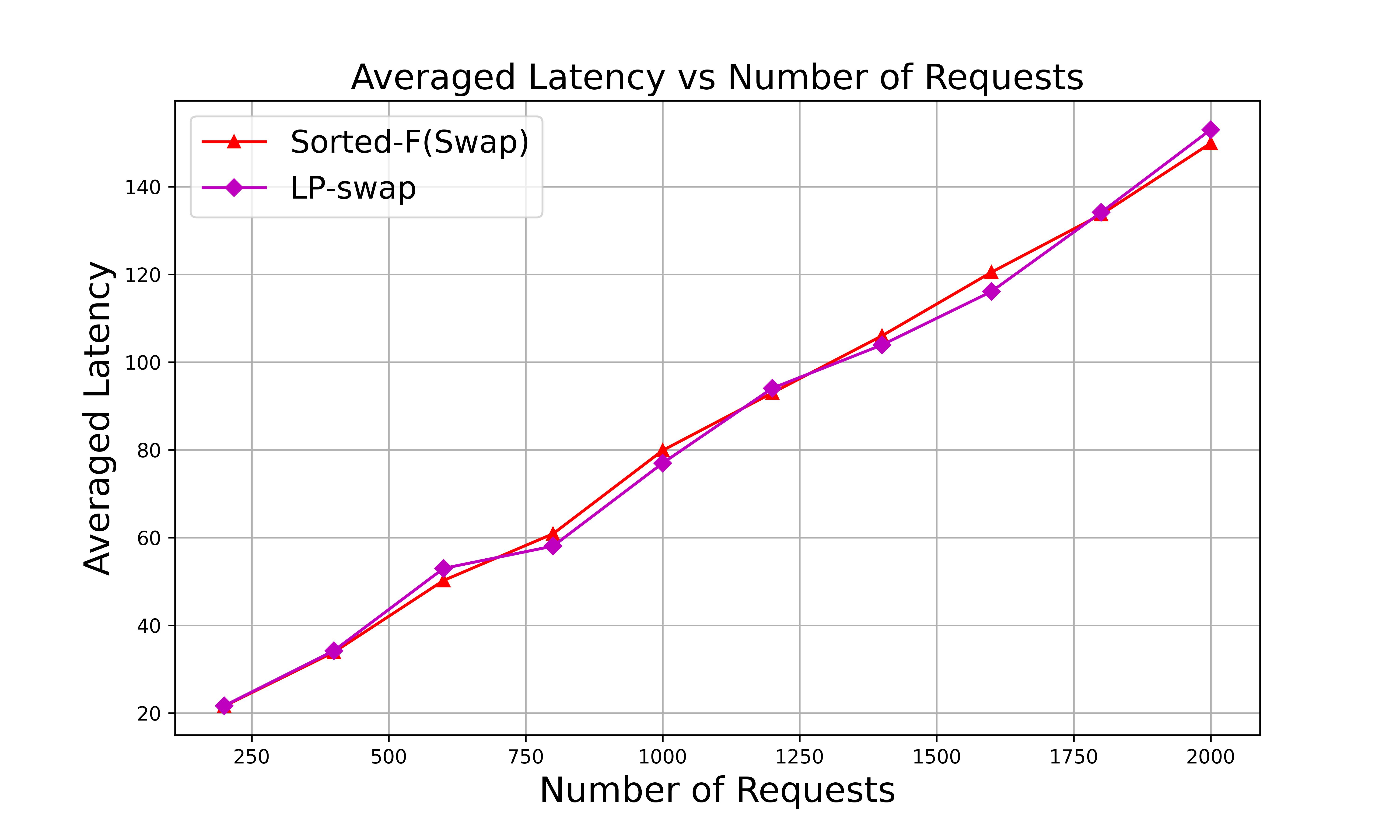}
\caption{Averaged latency between \sortF{} and \LPSwap}
\label{fig:resultappend}
\end{figure}

Figure~\ref{fig:sortg} examines the robustness of the F-metric to explicitly
weighting prefill length in the batch-selection objective.
At the compute-matched prefill weight, Sorted-G and \sortF{} have
similar reported mean latency across the tested settings.
Increasing the prefill weight places greater emphasis on input
length within the batch-selection metric and can degrade latency.
In particular, at $\alpha=1$, the metric becomes
$G(\mathcal X)=\sum_{i\in\mathcal X}(s_i+o_i)/|\mathcal X|^2$;
because it still depends on batch cardinality, it is not, in general,
equivalent to prioritizing individual requests by total sequence
length $s_i+o_i$.

\begin{figure}[!h]
\center
\includegraphics[width=0.9\textwidth]{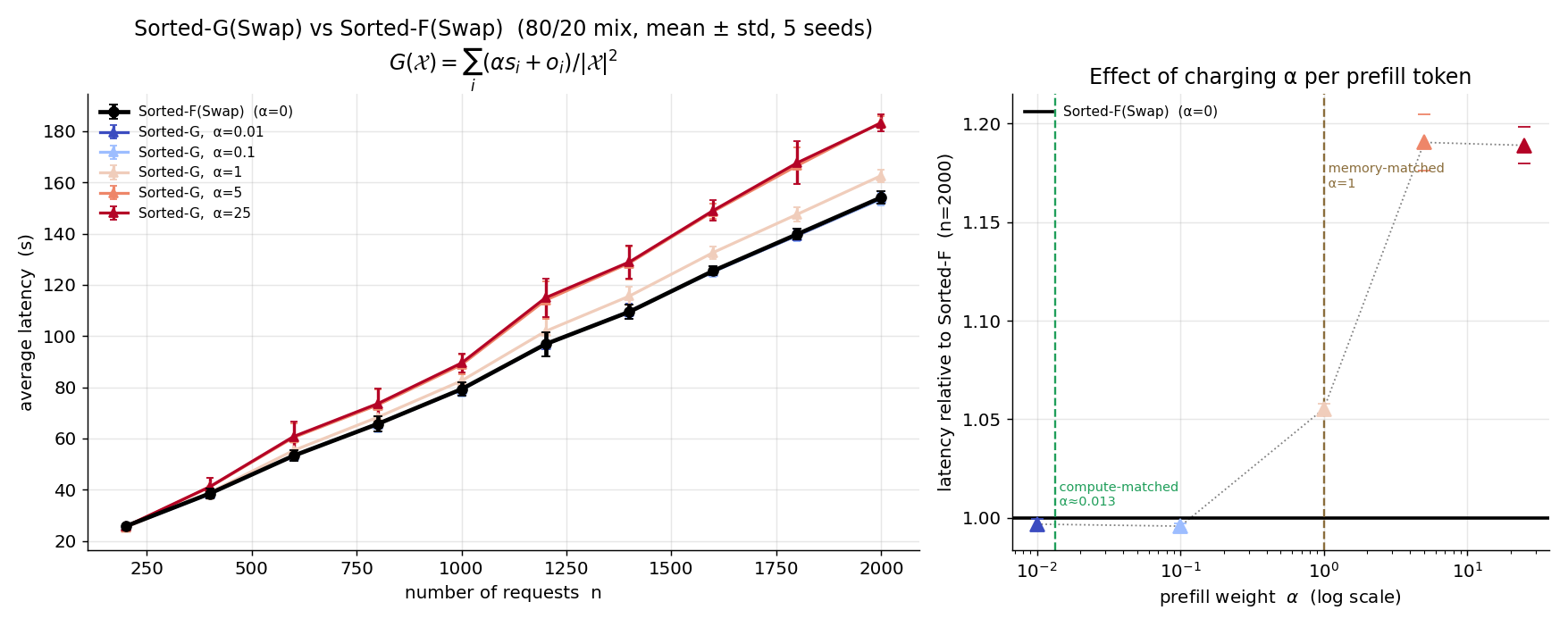}
\caption{Robustness to prefill weighting. Phase~1 uses the batch metric
$G(\mathcal X)=\sum_{i\in\mathcal X}(\alpha s_i+o_i)/|\mathcal X|^2$;
$\alpha=0$ recovers the $F$-metric. Latency is evaluated using the Vidur timing
model, including prefill and decode execution costs. Left: mean latency versus
workload size for the $80/20$ mixture, with error bars showing one standard
deviation over five seeds. Right: latency at $n=2000$ relative to Sorted-F as
$\alpha$ varies on a logarithmic scale. The compute-matched value
$\alpha\approx0.013$ reflects the relative per-token prefill and decode costs
in the timing model. At $\alpha=1$, the metric becomes
$G(\mathcal X)=\sum_{i\in\mathcal X}(s_i+o_i)/|\mathcal X|^2$.
It therefore continues to depend on batch cardinality and is not, in general,
equivalent to prioritizing individual requests by $s_i+o_i$ as in \mcsdff.
Sorted-G and Sorted-F have similar reported mean latency at the compute-matched
weight; larger weights yield latency increases of up to approximately $19\%$
in the tested settings. These observations concern empirical wall-clock
performance, not an extension of the factor-$40$ theorem.}
\label{fig:sortg}
\end{figure}

Figure~\ref{fig:tail_latency} complements the average-latency results by reporting tail-latency metrics under the 80/20 workload mixture. The figure shows that \sortF{} with Local Swap Search improves not only mean latency but also high-percentile latency relative to standard baselines, supporting its relevance for service-level objectives in LLM serving.

\begin{figure}[!h]
\center
\includegraphics[width=0.95\textwidth]{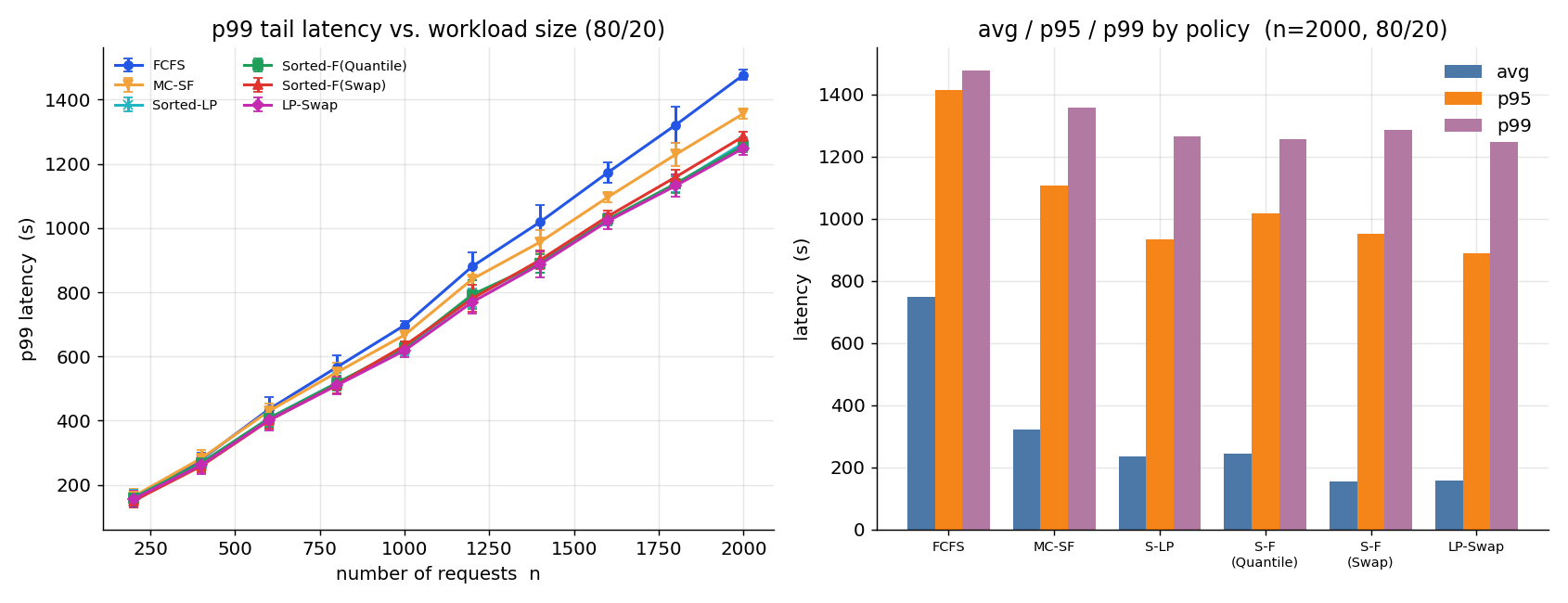}
\caption{Tail-latency comparison under the \(80/20\) workload mixture. Left: p99 latency as a function of workload size. Right: average, p95, and p99 latency by policy at \(n=2000\).}
\label{fig:tail_latency}
\end{figure}

Figure~\ref{fig:online_gap} summarizes the Poisson-arrival online experiment. The results show that \sortF{}-Online performs similarly to the online \mcsdf{} baseline when the system load is light, but its advantage grows in overloaded regimes where a backlog forms, consistent with the motivation for the offline/backlogged model.

\begin{figure}[!h]
\center
\includegraphics[width=0.95\textwidth]{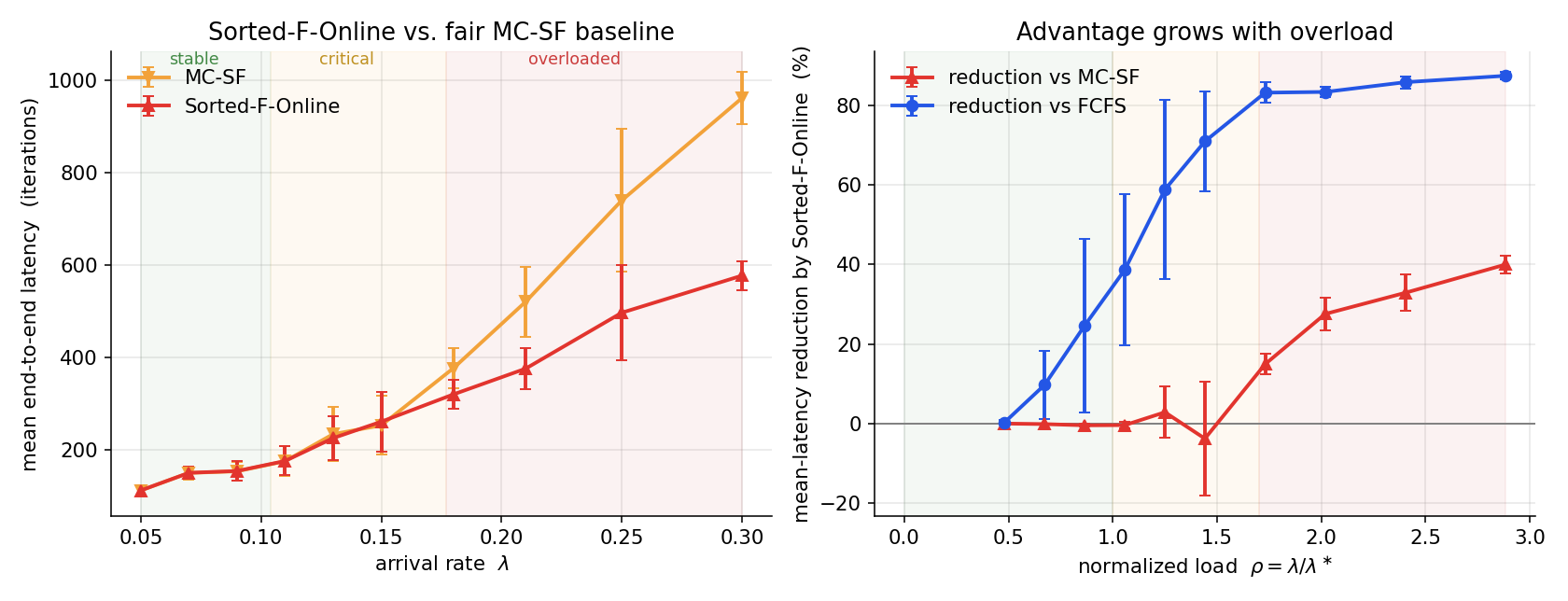}
\caption{Comparison between \sortF{}-Online and the online \mcsdf{} baseline. Left: mean latency under \mcsdf{} and \sortF{}-Online. Right: mean-latency reduction of \sortF{}-Online relative to \mcsdf{} and \texttt{FCFS} as a function of normalized load.}
\label{fig:online_gap}
\end{figure}

\end{document}